\documentclass[a4paper,10pt]{article}
\usepackage[latin1]{inputenc}
\usepackage[T1]{fontenc}
\usepackage{amsmath}
\usepackage{amsfonts}
\usepackage{amstext}
\usepackage{amsthm}
\usepackage[all]{xy}
\usepackage{geometry}
\usepackage{srcltx}
\usepackage{graphics}
\usepackage{epsfig}
\usepackage{subfigure}
\usepackage{slashed}
\usepackage{color}
\usepackage{amssymb}
\usepackage{srcltx}
\newtheorem{theorem}{Theorem}[section]
\newtheorem{lemma}[theorem]{Lemma}
\newtheorem{defi}[theorem]{Definition}

\newtheorem{cor}[theorem]{Corollary}

\DeclareMathOperator{\im}{im}

\DeclareMathOperator{\sfl}{sf}
\DeclareMathOperator{\diag}{diag}

\DeclareMathOperator{\sgn}{sgn}

\DeclareMathOperator{\Sp}{Sp}

\DeclareMathOperator{\re}{Re}

\title{Spectral flow, crossing forms and homoclinics of Hamiltonian systems}
\author{Nils Waterstraat}

\begin{document}
\date{}
\maketitle

\footnotetext[2]{N. Waterstraat was supported by the Berlin Mathematical School and the SFB 647 ``Space--Time--Matter''.}

\begin{abstract}
We prove a spectral flow formula for one-parameter families of Hamiltonian systems under homoclinic boundary conditions, which relates the spectral flow to the relative Maslov index of a pair of curves of Lagrangians induced by the stable and unstable subspaces, respectively. Finally, we deduce sufficient conditions for bifurcation of homoclinic trajectories of one-parameter families of nonautonomous Hamiltonian vector fields.    

\vskip1truecm
\centerline{Dedicated to Jacobo Pejsachowicz\quad * 11.10.1944} 
\end{abstract}

\section{Introduction}
We denote by $I:=[0,1]$ the unit interval and we consider for $\lambda\in I$ homoclinic solutions of Hamiltonian systems
 
\begin{equation}\label{Hamiltonian}
\left\{
\begin{aligned}
Ju'(t)+S_\lambda(t)u(t)&=0,\quad t\in\mathbb{R}\\
\lim_{t\rightarrow\pm\infty}u(t)&=0,
\end{aligned}
\right.
\end{equation}
where $S:I\times\mathbb{R}\rightarrow\mathcal{S}(\mathbb{R}^{2n})$ is a smooth family of symmetric matrices having uniform limits $S_\lambda(\pm\infty):=\lim_{t\rightarrow\pm\infty}S_\lambda(t)$, and 

\begin{align}\label{J}
J=\begin{pmatrix}
0&-I_n\\
I_n&0
\end{pmatrix}
\end{align}
is the standard symplectic matrix. Let us recall that the stable and unstable subspaces of \eqref{Hamiltonian} at $t_0\in\mathbb{R}$ are given by

\begin{align}\label{stableunstablesubspaces}
\begin{split}
E^s_\lambda(t_0)&=\{u(t_0):\, Ju'(t)+S_\lambda(t) u(t)=0,\, t\in\mathbb{R};\,u(t)\rightarrow 0,t\rightarrow\infty \}\\
E^u_\lambda(t_0)&=\{u(t_0):\, Ju'(t)+S_\lambda(t) u(t)=0,\, t\in\mathbb{R};\,u(t)\rightarrow 0,t\rightarrow-\infty  \}.
\end{split}
\end{align} 
Clearly, there is a non-trivial solution of \eqref{Hamiltonian} if and only if $E^s_\lambda(t_0)\cap E^u_\lambda(t_0)\neq\{0\}$ for some (and hence any) $t_0\in\mathbb{R}$.\\
The spectral flow is an integer-valued homotopy invariant for paths of selfadjoint Fredholm operators that was introduced by Atiyah, Patodi and Singer in \cite{AtiyahPatodi} in connection with spectral asymmetry and the $\eta$-invariant. Roughly speaking, it is the net number of eigenvalues (counted with multiplicities) which pass through zero in the positive direction when the parameter of the path travels along the unit interval. There are several different but equivalent definitions of the spectral flow with various degrees of generality that have appeared in the literature during the last decades. Here we just want to mention \cite{BoossDesuspension}, \cite{Floer}, \cite{Robbin-Salamon}, \cite{Specflow}, \cite{UnbSpecFlow}, \cite{Wahl}, which is probably far away from being exhaustive. In what follows, we use the definition of \cite{UnbSpecFlow} which applies to any gap-continuous path of (generally) unbounded selfadjoint Fredholm operators on a separable Hilbert space. The differential equations \eqref{Hamiltonian} induce such operators $\mathcal{A}_\lambda$, $\lambda\in I$, on $L^2(\mathbb{R},\mathbb{R}^{2n})$ having as domain $H^1(\mathbb{R},\mathbb{R}^n)$, and such that the kernel of $\mathcal{A}_\lambda$ is given by the solutions of the corresponding equation \eqref{Hamiltonian}. We explain below that the spectral flow of the resulting path $\mathcal{A}$ is defined, and heuristically, it counts in an oriented way the instants $\lambda\in I$ for which the equations \eqref{Hamiltonian} have non-trivial solutions.\\  
Let us recall that $\mathbb{R}^{2n}$ is a symplectic space with respect to the symplectic form $\omega$ induced by \eqref{J}, i.e. $\omega(u,v)=\langle Ju,v\rangle_{\mathbb{R}^{2n}}$, $u,v\in \mathbb{R}^{2n}$. It is readily seen that $\omega(v(t_0),w(t_0))$ vanishes for all $t_0\in\mathbb{R}$ if $v$ and $w$ solve the differential equation $Ju'(t)+S_\lambda(t) u(t)=0$, $t\in\mathbb{R}$, and decay to zero at least in one of the limits $t\rightarrow\pm\infty$. Consequently, $E^s_\lambda(t_0)$ and $E^u_\lambda(t_0)$ are isotropic subspaces of $\mathbb{R}^{2n}$ and since they are of dimension $n$ under common assumptions that we introduce below, it follows that they actually are Lagrangian. The Maslov index assigns to any pair of paths of Lagrangian subspaces of a symplectic vector space an integer, which heuristically, counts non-trivial intersections between the spaces. There are several different constructions of the Maslov index in the literature, and here we just refer to \cite{Arnold}, \cite{Cappel}, and in particular to \cite{Robbin-SalamonMAS}, which we use below for defining the relative Maslov index $\mu_{Mas}(E^u_\cdot(t_0),E^s_\cdot(t_0))$ of the curves of Lagrangian subspaces induced by $E^s_\lambda(t_0)$ and $E^u_\lambda(t_0)$, $\lambda\in I$, for any fixed $t_0\in\mathbb{R}$.\\
Our main theorem shows that the spectral flow of the path $\mathcal{A}$ of unbounded selfadjoint Fredholm operators on $L^2(\mathbb{R},\mathbb{R}^{2n})$ induced by the Hamiltonian systems \eqref{Hamiltonian} coincides with the Maslov index $\mu_{Mas}(E^u_\cdot(t_0),E^s_\cdot(t_0))$ of the evolution of the unstable and stable subspaces \eqref{stableunstablesubspaces}.\\
Let us point out that a related result was proven before under different assumptions by Chen and Hu in \cite{Hu}. They suppose that the limits $S_\lambda(\pm\infty)$ are a single constant matrix, which allows to avoid some assumptions on the Hamiltonian systems \eqref{Hamiltonian} that we will require below. Here, however, we follow Pejsachowicz' setting from \cite{Jacobo} (cf. also \cite{Jacobohomoclinics}), who proved our main theorem under the additional assumption that $S_0=S_1:\mathbb{R}\rightarrow\mathcal{S}(\mathbb{R}^{2n})$, i.e. when the parameter space is the unit circle $S^1$ instead of the unit interval $I$. In this case, the spectral flow of the corresponding closed path $\mathcal{A}$ is equal to the relative Maslov index $\mu_{Mas}(E^s_\cdot(+\infty),E^u_\cdot(-\infty))$, where $E^s_\lambda(\pm\infty)$ and $E^u_\lambda(\pm\infty)$, $\lambda\in I$, denote the stable and unstable subspaces of the equations $Ju'(t)+S_\lambda(\pm\infty)u(t)=0$, $t\in\mathbb{R}$ (cf. \eqref{stabinf}, \eqref{unstabinf}). Of course, since these equations are autonomous, $E^s_\lambda(\pm\infty)$ and $E^u_\lambda(\pm\infty)$ can be computed easily from $S_\lambda(\pm\infty)$.  
Pejsachowicz' argument is inspired by the Atiyah-Singer index theorem and makes essentially use of the fact that the unit circle is topologically non-trivial. He introduces a symbol for Hamiltonian systems depending on the lower order terms and uses homotopy theory to conclude the equality of the spectral flow and the Maslov index. Consequently, these methods cannot be transfered to the setting that we are considering here, and in addition, our main theorem shows that the spectral flow of $\mathcal{A}$ will in general depend on the stable and unstable subspaces of the original non-autonomous equations $Ju'(t)+S_\lambda(t)u(t)=0$, $t\in\mathbb{R}$.\\
The proof of our main theorem is purely analytical and has some difficulties in its own. First, we use crossing forms for the computation of the spectral flow for paths $\mathcal{A}=\{\mathcal{A}_\lambda\}$ of selfadjoint Fredholm operators on a Hilbert space $H$ which have a constant dense domain $\mathcal{D}(\mathcal{A}_\lambda)=W\subset H$. We assume that $W$ is a Hilbert space in its own right and that the canonical inclusion $\iota:W\rightarrow H$ is continuous. The concept of crossing forms for the computation of the spectral flow was introduced by Robbin and Salamon in \cite{Robbin-Salamon} under the additional assumption that $\iota$ is compact, so that in particular the spectra of the operators $\mathcal{A}_\lambda$ are discrete. Later on, Fitzpatrick, Pejsachowicz and Recht showed in \cite{Specflow} that the spectral flow can also be computed by crossing forms for paths of bounded selfadjoint Fredholm operators, i.e. if $W=H$. However, both settings do not apply to our situation, since $W=H^1(\mathbb{R},\mathbb{R}^{2n})$ does neither coincide with $H=L^2(\mathbb{R},\mathbb{R}^{2n})$ nor is the embedding compact, and our first purpose is to establish crossing forms in the case of general spaces $W$ and $H$ as above. Second, from the theory of crossing forms, we will see that we can assume that our path $\mathcal{A}_\lambda$ is invertible except at one instant $\lambda_0\in I$ and that the spectral flow can be computed by means of a quadratic form on the kernel of $\mathcal{A}_{\lambda_0}$. This allows to reduce the computation of the spectral flow to operators that are defined on some $L^2$ space on a compact interval, which however, has to be done carefully in view of subsequent steps of our proof. Third, an originally unexpected term occurs in the middle of our proof that we need to treat by using perturbation methods \cite{Kato}. Finally, the equality of the spectral flow and the Maslov index can be traced back to the tools that were already introduced by Robbin and Salamon in \cite{Robbin-Salamon}.\\
Let us say a few words on former work on spectral flow formulas involving the Maslov index, where however, we do not claim to be exhaustive. The first theorem identifying the spectral flow of a differential operator with a Maslov index of which we are aware was proven by Salamon and Zehnder in \cite{Salamon-Zehnder} for periodic Hamiltonian systems. Here the curves of Lagrangian subspaces for the Maslov index are induced by the monodromy matrices (cf. \cite[Rem. 5.4]{Robbin-SalamonMAS}). A similar result was later shown by Fitzpatrick, Pejsachowicz and Recht in \cite{SFLPejsachowiczII} in the study of bifurcation theory, and recently the author generalised these theorems in \cite{Hamiltonian} to families of periodic Hamiltonian system by considering index bundles for families of selfadjoint Fredholm operators. Finally, let us mention in passing that spectral flow formulas for Hamiltonian systems on a compact interval under non-periodic boundary conditions can be found, for example, in \cite{Robbin-Salamon} and \cite{Cappel}. More general, the spectral flow for boundary value problems of Dirac operators and its relation to the Maslov index in symplectic Hilbert spaces has been studied extensively. Here we only mention \cite{NicolaescuDuke}, \cite{NicolaescuMem}, \cite{KirkLesch} and \cite{BBBZhu}.\\  
Pejsachowicz applies his spectral flow formula in \cite{Jacobo} to bifurcation of homoclinic solutions of families of nonlinear Hamiltonian systems parametrised by the circle, where he uses a bifurcation theorem for critical points of strongly indefinite families of functionals from his joint work \cite{Specflow} with Fitzpatrick and Recht. In our final Section \ref{sect:bifurcation}, we deduce from our spectral flow formula and the recent work \cite{JacBifIch} of Pejsachowicz and the author the following assertion: Let $\mathcal{H}:I\times\mathbb{R}\times\mathbb{R}^{2n}\rightarrow\mathbb{R}$ be a continuous map such that $\mathcal{H}_\lambda:=\mathcal{H}(\lambda,\cdot,\cdot):\mathbb{R}\times\mathbb{R}^{2n}\rightarrow\mathbb{R}$ is $C^2$ for all $\lambda\in I$ and a usual growth condition is satisfied (cf. \eqref{Hgrowth} below). We consider the family of systems

\begin{equation}\label{introHamnon}
\left\{
\begin{aligned}
Ju'(t)+\nabla_u \mathcal{H}_\lambda(t,u(t))&=0,\quad t\in\mathbb{R}\\
\lim_{t\rightarrow\pm\infty}u(t)&=0,
\end{aligned}
\right.
\end{equation}
where $\nabla_u$ denotes the gradient with respect to the variable $u\in\mathbb{R}^{2n}$ and we assume that $\nabla_u \mathcal{H}_\lambda(t,0)=0$ for all $\lambda\in I$. An instant $\lambda_0\in I$ is called a bifurcation point if there exists a sequence $(\lambda_n,u_{\lambda_n})$ such that $u_{\lambda_n}\not\equiv0$ satisfies the equation \eqref{introHamnon} for $\lambda_n$ and $u_{\lambda_n}$ tends to zero for $\lambda\rightarrow \lambda_0$ in the $C^1$-topology. The linearisations of the equations \eqref{introHamnon} are of the form \eqref{Hamiltonian}, and our bifurcation theorem states that non-trivial intersections of the stable and unstable subspaces $E^u_\lambda(0)$ and $E^s_\lambda(0)$ of \eqref{Hamiltonian} which give a non-vanishing Maslov index $\mu_{Mas}(E^u_\cdot(0),E^s_\cdot(0))$ cause bifurcation of homoclinics for \eqref{introHamnon}.\\
Our construction of the Maslov index is based on Abbondandolo and Majer's investigations on infinite dimensional stable and unstable subspaces \cite{Alberto}. Accordingly, we shall extend our theory in a subsequent project to Hamiltonian partial differential equations as in \cite{Bartsch}.\\ 
The paper is structured as follows: In the second section we adapt the definition of the spectral flow from \cite{UnbSpecFlow} to a class of (generally unbounded) selfadjoint Fredholm operators that is suitable for studying homoclinics of Hamiltonian systems. Moreover, we generalise a perturbation theorem of Robbin and Salamon from \cite{Robbin-Salamon} to this type of operators and show that the spectral flow can be computed by crossing forms. The third section briefly recalls the Maslov index for pairs of paths of Lagrangian subspaces in $\mathbb{R}^{2n}$ from \cite{Robbin-SalamonMAS}. In Section 4 we state our spectral flow formula and prove it in eight steps. Finally, in the fifth section we consider nonlinear Hamiltonian systems and apply our spectral flow formula to bifurcation of homoclinic trajectories. The paper has three appendices. Appendix A deals with elementary perturbation theory of quadratic forms and summarises well known facts that we use throughout the paper. Appendix B proves a technical lifting lemma for maps in the Lagrangian Grassmannian that we need in the proof of our spectral flow formula. Finally, Appendix C contains the proofs of two rather technical assertions regarding the spectral flow, which we separate from the second section for a better readability.\\      
At last, let us introduce some notation that we shall use henceforth without further reference. We have already mentioned that $I$ stands for the unit interval, however, the similar symbol $I_X$ will denote the identity operator on a space $X$ and we set for simplicity $I_k:=I_{\mathbb{R}^{k}}$ for $k\in\mathbb{N}$. We consider throughout smooth families $\{\Psi_\lambda\}_{\lambda\in I}$, where each $\Psi_\lambda$ is a matrix valued function on the real line. We denote by $\Psi'_\lambda(t)$ the derivative with respect to the variable $t\in\mathbb{R}$, whereas $\dot{\Psi}_\lambda(t)$ stands for the derivative with respect to the parameter $\lambda\in I$. We denote by $\Psi_\lambda(t)^\ast$ the transpose of $\Psi_\lambda(t)$. Finally, let us point out that $\lambda$ will usually be a parameter in $I$, except in Section \ref{sect:sfl} where it stands for an eigenvalue of a linear operator.  

\subsubsection*{Acknowledgements}
Parts of this work were introduced in a lecture series that the author gave at the Universit\`a degli studi di Torino in Italy in spring 2013. We are grateful to Anna Capietto and Alessandro Portaluri for inviting us to give these lectures, and to the audience for several valuable questions. Moreover, we would like to thank the anonymous referee for his careful reading of our manuscript and for pointing out to us an unnecessary assumption in our main theorems in Section 2, which has been removed.

%%%%%%%%%%%%%%%%%%%%%%%%%%%%%%%%%%%%%%%%%%%%%%%%%%%%%%%%%%%%%%%%%%%%%%%%%%%%%%%%%%%%%%%%%%%%%%%%%%%%%%%%%%%%%%%%%%%%%%%%%%%%%%%%%%%%%%%%%%%%%%%%%%%%%%%%%%%%%%%%%%%%%%%%%%%%%%%%%%%%%%%%%%%%%%%%%%%%%%%%%%%%%%%%%%%%%%%%%%%%%%%%%%%%%%%%%%%%%%%%%%%%%%%%%%%%%%%%%%%%%%%%%%%%%%%%%%%%%%%%%%%%%%%%%%%%%%%%%%%%%%%%%%%%%%%%%%%%%%%%%%%%%%%%%%%%%%%%%%%%%%%%%%%%%%%%%%%%%%%%%%

\section{The spectral flow and crossing forms}\label{sect:sfl}
Let $W, H$ be real Hilbert spaces with a dense and continuous inclusion $W\hookrightarrow H$. We denote by $\mathcal{L}(W,H)$ the Banach space of all bounded operators, by $GL(W,H)\subset\mathcal{L}(W,H)$ the open subset of all invertible elements and by $\mathcal{S}(W,H)\subset\mathcal{L}(W,H)$ the closed subset of all operators which are selfadjoint when regarded as operators on $H$ with dense domain $W$. Finally, we denote by $\mathcal{FS}(W,H)$ the set of all Fredholm operators in $\mathcal{S}(W,H)$, and we recall that an operator in $\mathcal{S}(W,H)$ is Fredholm if and only if its kernel is of finite dimension and its image is closed.  Note that if $T$ is a selfadjoint operator on a Hilbert space $H$ having the domain $W$, then $W$ equipped with the graph norm of $T$ is continuously embedded in $H$. Consequently, for a single selfadjoint operator $T$ on $H$, the existence of a continuously embedded subspace $W$ is clear by definition.\\
Let us point out that there are two important special cases:

\begin{enumerate}
	\item[i)] If $W=H$, we shorten notation as usual and write e.g. $\mathcal{L}(H):=\mathcal{L}(H,H)$. Note that in this case all operators in $\mathcal{S}(H)$ are bounded, whereas elements of $\mathcal{S}(W,H)$ are never bounded as selfadjoint operators on $H$ if $W\neq H$.
	\item[ii)] The inclusion $W\hookrightarrow H$ is compact. This assumption in particular implies that elements of $\mathcal{S}(W,H)$ have compact resolvents, i.e. the linear operator $(\lambda-T)^{-1}:H\rightarrow H$ is compact if $\lambda-T\in GL(W,H)$ for a scalar $\lambda$. 
\end{enumerate}   
In what follows we will use without further reference that if $T\in\mathcal{S}(W,H)$, then $T+S\in\mathcal{S}(W,H)$ for all $S\in\mathcal{S}(H)$.

%%%%%%%%%%%%%%%%%%%%%%%%%%%%%%%%%%%%%%%%%%%%%%%%%%%%%%%%%%%%%%%%%%%%%%%%%%%%%%%%%%%%%%%%%%%%%%%%%%%%%%%%%%%%%%%%%%%%%%%%%%%%%%%%%%%%%%%%%%%%%%%%%%%%%%%%%%%%%%%%%%%%%%%%%%%%%%%%%%%%%%%%%%%%%%%%%%%%%%%%%%%%%%%%%%%%%%%%%%%%%%%%%%%%%%%%%%%%%%%%%%%%%%%%%%%%%%%%%%%%%%%%%%%%%%%%%%%%%%%%%%%%%%%%%%%%%%%%%%%%%%%%%%%%%%%%%%%%%%%%%%%%%%%%%%%%%%%%%%%%%%%%%%%%%%%%%%%%%%%%%%

\subsection{The spectral flow}
The aim of this section is to define the spectral flow for paths in $\mathcal{FS}(W,H)$, where we essentially follow \cite{Phillips} in which the case $W=H$ was considered. In \cite{UnbSpecFlow} the spectral flow was constructed in the more general case of paths of unbounded selfadjoint Fredholm operators that are continuous in the gap-topology and in particular may have varying domains. Let us point out that every path in $\mathcal{FS}(W,H)$ is also continuous with respect to the gap-topology (cf. \cite[Prop. 2.2]{LeschSpecFlowUniqu}), so that we introduce here just a special case of the former work \cite{UnbSpecFlow}. However, if we restrict to paths in $\mathcal{FS}(W,H)$, which is completely sufficient for the applications that we have in mind, the theory turns out to be as straightforward as in the bounded case. Moreover, as we show in the subsequent section, the computation of the spectral flow by means of crossing forms, which was figured out by Robbin and Salamon in \cite{Robbin-Salamon} under assumption ii) and by Fitzpatrick, Pejsachowicz and Recht in \cite{Specflow} under assumption i) from above, holds in $\mathcal{FS}(W,H)$ for general $W$ and $H$.\\  
We denote for $T\in\mathcal{S}(W,H)$ by 

\[\sigma(T)=\{\lambda\in\mathbb{R}:\,\lambda-T\notin GL(W,H)\}\]
the spectrum of $T$ and we recall that this is a non-empty closed subset of the real line, which is bounded if and only if $W=H$. Moreover, if $W\hookrightarrow H$ is compact, then $\sigma(T)$ is discrete. For $a,b\notin\sigma(T)$, we set

\begin{align}\label{specproj}
P_{[a,b]}(T)=\re\left(\frac{1}{2\pi i}\int_\Gamma{(\lambda-T^\mathbb{C})^{-1}\,d\lambda}\right),
\end{align}
where $\Gamma$ is the circle of radius $\frac{b-a}{2}$ around $\frac{a+b}{2}$ and $\re$ denotes the real part of an operator on a complexified Hilbert space (cf. eg. \cite{domainshrinking}). Let us recall from \cite[\S XV.2]{GohbergClasses}, that if $[a,b]\cap\sigma(T)$ consists solely of isolated eigenvalues of finite type, then

\[\im P_{[a,b]}(T)=\bigoplus_{\lambda\in(a,b)}{\ker(\lambda-T)}.\]
Even though the following lemma is folklore, we include it for the reader's convenience since it plays a decisive role in the definition of the spectral flow.

\begin{lemma}\label{0isolated}
If $T\in\mathcal{FS}(W,H)$, then either $0$ belongs to the resolvent set of $T$ or it is an isolated eigenvalue of finite multiplicity.
\end{lemma}

\begin{proof}
Since the set of Fredholm operators is open in $\mathcal{L}(W,H)$, there exists $\varepsilon>0$ such that $\lambda-T$ is Fredholm of index $0$ for all $|\lambda|<\varepsilon$. Consequently, for $|\lambda|<\varepsilon$, either $\lambda-T$ is invertible or it has a finite dimensional kernel. It remains to show that $0$ is not a limit point of $\sigma(T)$.\\
Since $T$ is selfadjoint and Fredholm, there is an orthogonal decomposition $H=\ker T\oplus\im T$. If we set $W'=\iota^{-1}(\im T)$, where $\iota:W\hookrightarrow H$ denotes the canonical inclusion, we obtain a decomposition $W=\ker T\oplus W'$ into closed subspaces of $W$. The restriction $T'$ of $T$ to $W'$ is an isomorphism onto the closed subspace $\im T$ of $H$. Since $GL(W',\im T)\subset\mathcal{L}(W',\im T)$ is open, there exists $\kappa>0$ such that $\sigma(T')\cap(-\kappa,\kappa)=\emptyset$. Let now $0<|\lambda|<\min\{\kappa,\varepsilon\}$ and let $u=u_1+u_2\in\ker T\oplus W'=W$ be an element of $\ker(\lambda-T)$. Then $0=\lambda u-Tu=\lambda u_1+\lambda u_2-Tu_2$ and so 

\[\lambda u_1=(T-\lambda)u_2=(T'-\lambda)u_2.\]
Since the left hand side of this equation is in $\ker T$ and the right hand side is in $\im T=(\ker T)^\perp$, we conclude that $\lambda u_1=(T'-\lambda)u_2=0$. Finally, from $\lambda\neq 0$ and the injectivity of $T'-\lambda$, it follows that $u=u_1+u_2=0$ and so $\lambda\notin\sigma(T)$.    
\end{proof}

Next, we show that the projections \eqref{specproj} depend continuously on the operator $T$.

\begin{lemma}\label{lemma-contproj}
For $a,b\in\mathbb{R}$, $a<b$, the set

\begin{align*}
\Omega_{[a,b]}=\{T\in\mathcal{S}(W,H):a,b\notin\sigma(T)\}\subset\mathcal{S}(W,H)
\end{align*}
is open, and the map

\begin{align}\label{specprojII}
\Omega&_{[a,b]}\rightarrow\mathcal{L}(H),\quad T\mapsto P_{[a,b]}(T)
\end{align}
is continuous.
\end{lemma}

\begin{proof}
The fist assertion is an immediate consequence of the openness of $GL(W,H)$. In order to show the continuity of the map \eqref{specprojII}, we note at first that for any $T,S\in\Omega_{[a,b]}$

\begin{align}\label{ClosedOps-align-contproj}
\|P_{[a,b]}(T)-P_{[a,b]}(S)\|\leq\frac{b-a}{2}\max_{\lambda\in\Gamma}\|(\lambda-T^\mathbb{C})^{-1}-(\lambda-S^\mathbb{C})^{-1}\|,
\end{align}
where $\Gamma$ denotes the circle of radius $\frac{b-a}{2}$ around $\frac{a+b}{2}$. Moreover, the map

\begin{align}\label{rescont}
\Gamma\times\Omega_{[a,b]}\ni(\lambda,T)\mapsto (\lambda-T^\mathbb{C})^{-1}\in\mathcal{L}(H^\mathbb{C})
\end{align}
is continuous. Let now $T\in\Omega_{[a,b]}$ and $\varepsilon>0$. By the continuity of \eqref{rescont}, for any $\lambda'\in\Gamma$ there exists $\delta(\lambda')>0$  such that 

\begin{align*}
\|(\lambda-S^\mathbb{C})^{-1}-(\lambda'-T^\mathbb{C})^{-1}\|<\frac{\varepsilon}{b-a}
\end{align*}
for all $\lambda\in U(\lambda',\delta(\lambda')):=\{\lambda\in\Gamma:|\lambda-\lambda'|<\delta(\lambda')\}$ and $\|S-T\|<\delta(\lambda')$. We now take $\lambda_1,\ldots,\lambda_n\in\Gamma$ such that $\bigcup^n_{i=1}{U(\lambda_i,\delta(\lambda_i))}=\Gamma$ and we set $\delta:=\min_{1\leq i\leq n}\delta(\lambda_i)$.\\
Now, for any $\lambda\in\Gamma$ there exists $1\leq i\leq n$ such that $\lambda\in U(\lambda_i,\delta(\lambda_i))$ and hence we obtain for $S\in\Omega_{[a,b]}$,  $\|S-T\|<\delta$,

\begin{align*}
\|(\lambda-T^\mathbb{C})^{-1}-(\lambda-S^\mathbb{C})^{-1}\|&\leq \|(\lambda-T^\mathbb{C})^{-1}-(\lambda_i-T^\mathbb{C})^{-1}\|+\|(\lambda_i-T^\mathbb{C})^{-1}-(\lambda-S^\mathbb{C})^{-1}\|\\
&<\frac{2\varepsilon}{b-a}.
\end{align*}
We conclude by \eqref{ClosedOps-align-contproj} that $\|P_{[a,b]}(T)-P_{[a,b]}(S)\|<\varepsilon$ for all $S\in\Omega_{[a,b]}$ such that $\|S-T\|<\delta$.
\end{proof}

The following corollary paves the way for the definition of the spectral flow.

\begin{cor}\label{lem:sflneighbourhood}
Let $T_0\in\mathcal{FS}(W,H)$ be fixed.

\begin{enumerate}
	\item[i)] There exists a positive real number $a\notin\sigma(T_0)$ and an open connected neighbourhood $N\subset\mathcal{FS}(W,H)$ of $T_0$ such that $\pm a\notin\sigma(T)$ for all $T\in N$,
	
	\begin{align*}
	N\rightarrow\mathcal{L}(H),\quad T\mapsto P_{[-a,a]}(T)
	\end{align*}
is continuous, and the projection $P_{[-a,a]}(T)$ has constant finite rank for all $T\in N$.
\item[ii)] If $N$ is a neighbourhood as in i) and $-a\leq c<d\leq a$ are such that $c,d\notin\sigma(T)$ for all $T\in N$, then $T\mapsto P_{[c,d]}(T)$ is continuous on $N$. Moreover, the rank of $P_{[c,d]}(T)$, $T\in N$, is finite and constant. 
\end{enumerate}
\end{cor}

\begin{proof}
By Lemma \ref{0isolated} there exists $a>0$ such that $[-a,a]\cap\sigma(T_0)\subset\{0\}$. Now we take as $N$ the connected component of $T_0$ in $\Omega_{[-a,a]}$. Since $\dim\im P=\dim\im Q$ for any two projections $P,Q\in\mathcal{L}(H)$ such that $\|P-Q\|<1$ (cf. \cite[Lemma II.4.3]{GohbergClasses}), the rank of $P_{[-a,a]}(T)$ is locally constant on $\Omega_{[-a,a]}$. Consequently, $\dim\im P_{[-a,a]}(T)=\dim\ker T_0$ for all $T\in N$. The remaining assertions are immediate consequences of the previous lemma.
\end{proof}

It is worth to point out that if $c,d$ are as in the previous corollary, then $\sigma(T)\cap[c,d]$, $T\in N$, consists of a finite number of eigenvalues of finite multiplicity (cf. \cite[\S XV.2]{GohbergClasses}), and consequently

\begin{align}\label{projdim}
\im P_{[c,d]}(T)=E_{[c,d]}(T):=\bigoplus_{\lambda\in[c,d]}{\ker(\lambda-T)}.
\end{align}
Let now $\mathcal{A}:I\rightarrow\mathcal{FS}(W,H)$ be a continuous path. By Corollary \ref{lem:sflneighbourhood}, we conclude that for every $t\in I$ there exists $a>0$ and an open connected neighbourhood $N_{t,a}\subset\mathcal{FS}(W,H)$ of $\mathcal{A}_t$ such that $\pm a\notin\sigma(T)$ for all $T\in N_{t,a}$ and the map

\begin{align*}
N_{t,a}\rightarrow\mathcal{L}(H),\quad T\mapsto P_{[-a,a]}(T)
\end{align*} 
is continuous. Moreover, all $P_{[-a,a]}(T)$, $T\in N_{t,a}$, have the same finite rank. Now the counterimages of the $N_{t,a}$ under $\mathcal{A}$ define an open covering of the unit interval and, by using the Lebesgue number of this covering, we can find $0=t_0\leq t_1\leq\ldots\leq t_n=1$ and $a_i>0$, $i=1,\ldots n$, such that the maps

\begin{align*}
[t_{i-1},t_i]\ni t\mapsto P_{[-a_i,a_i]}(\mathcal{A}_t)\in\mathcal{L}(H)
\end{align*}
are continuous and have constant finite rank. We define the spectral flow of $\mathcal{A}:I\rightarrow\mathcal{FS}(W,H)$ by (cf. \eqref{projdim})

\begin{align}\label{IndPre-align-specflow}
\sfl(\mathcal{A})=\sum^n_{i=1}{\left(\dim E_{[0,a_i]}(\mathcal{A}_{t_i})-\dim E_{[0,a_i]}(\mathcal{A}_{t_{i-1}})\right)}.
\end{align}
The reader can find a proof of the well-definedness in Appendix C, as well as a proof of the following homotopy invariance property iv). The other assertions of Lemma \ref{IndPre-lemma-sflbasicprop} are immediate consequences of the definition \eqref{IndPre-align-specflow}.

\begin{lemma}\label{IndPre-lemma-sflbasicprop}
\begin{enumerate}
\item[i)] If $\mathcal{A}^1,\mathcal{A}^2:I\rightarrow\mathcal{FS}(W,H)$ are two paths such that $\mathcal{A}^2_0=\mathcal{A}^1_1$, then

	\begin{align*}
	\sfl(\mathcal{A}^1\ast\mathcal{A}^2)=\sfl(\mathcal{A}^1)+\sfl(\mathcal{A}^2).
	\end{align*}
\item[ii)] If $\mathcal{A}:I\rightarrow\mathcal{FS}(W,H)$ is continuous and $\mathcal{A}'$ is defined by $\mathcal{A}'_t=\mathcal{A}_{1-t}$, then

\begin{align*}
\sfl(\mathcal{A}')=-\sfl(\mathcal{A}).
\end{align*}

\item[iii)] If $\mathcal{A}:I\rightarrow \mathcal{FS}(W,H)$ is continuous and $\mathcal{A}_t$ invertible for all $t\in I$, then $\sfl(\mathcal{A})=0$.

\item[iv)] Let $h:I\times I\rightarrow\mathcal{FS}(W,H)$ be a continuous map such that $h(I\times\partial I)\subset GL(W,H)$. Then

\[\sfl(h(0,\cdot))=\sfl(h(1,\cdot)).\]
\end{enumerate}
\end{lemma}

%%%%%%%%%%%%%%%%%%%%%%%%%%%%%%%%%%%%%%%%%%%%%%%%%%%%%%%%%%%%%%%%%%%%%%%%%%%%%%%%%%%%%%%%%%%%%%%%%%%%%%%%%%%%%%%%%%%%%%%%%%%%%%%%%%%%%%%%%%%%%%%%%%%%%%%%%%%%%%%%%%%%%%%%%%%%%%%%%%%%%%%%%%%%%%%%%%%%%%%%%%%%%%%%%%%%%%%%%%%%%%%%%%%%%%%%%%%%%%%%%%%%%%%%%%%%%%%%%%%%%%%%%%%%%%%%%%%%%%%%%%%%%%%%%%%%%%%%%%%%%%%%%%%%%%%%%%%%%%%%%%%%%%%%%%%%%%%%%%%%%%%%%%%%%%%%%%%%%%%%%%

\subsection{Regular crossings and crossing forms}
In this section we discuss a method for computing the spectral flow which was introduced in \cite{Robbin-Salamon} and \cite{Specflow}, respectively. Here we follow essentially the lines of Robbin and Salamon \cite{Robbin-Salamon}, however, we will not assume that $W$ is compactly embedded in $H$, so that in our case the spectra of the operators are not necessarily discrete. This requires some modifications in the proofs of \cite{Robbin-Salamon}, however, since spectra of selfadjoint Fredholm operators cannot accumulate at $0$ (cf. Lemma \ref{0isolated}), we can apply the former arguments locally in a neighbourhood of $0$, which suffices for the computation of the spectral flow. Finally, let us point out that we in particular allow $W$ to coincide with $H$, in which case Theorem \ref{thm:crossing} below was proven by Fitzpatrick, Pejsachowicz and Recht in \cite{Specflow}.\\
From now on we assume that $\mathcal{A}:I\rightarrow\mathcal{FS}(W,H)\subset\mathcal{L}(W,H)$ is a continuously differentiable path. We denote by $\dot{\mathcal{A}}_{t_0}$ the derivative of $\mathcal{A}$ with respect to the parameter $t\in I$ at $t_0$. 

\begin{defi}
An instant $t_0\in I$ is called a \textit{crossing} if $\ker\mathcal{A}_{t_0}\neq 0$. The \textit{crossing form} at 
$t_0$ is the quadratic form defined by

\[\Gamma(\mathcal{A},t_0):\ker\mathcal{A}_{t_0}\rightarrow\mathbb{R},\,\,\Gamma(\mathcal{A},t_0)[u]=\langle\dot{\mathcal{A}}_{t_0}u,u\rangle_H.\]
A crossing $t_0$ is called \textit{regular}, if $\Gamma(\mathcal{A},t_0)$ is non-degenerate.
\end{defi} 

The aim of this section is to prove the following two theorems.

\begin{theorem}\label{thm:Sard}
There exists $\varepsilon>0$ such that 
\begin{itemize}
	\item[i)] $\mathcal{A}+\delta\, I_H$ is a path in $\mathcal{FS}(W,H)$ for all $|\delta|<\varepsilon$;
	\item[ii)] $\mathcal{A}+\delta\, I_H$ has only regular crossings for almost every $\delta\in(-\varepsilon,\varepsilon)$.
\end{itemize}
\end{theorem}

The second theorem shows that the spectral flow of $\mathcal{A}$ can be easily computed if all crossings are regular. We refer to Appendix A for the used notations.

\begin{theorem}\label{thm:crossing}
If $\mathcal{A}$ has only regular crossings, then they are finite in number and

\begin{align}\label{sflcross}
\sfl(\mathcal{A})=-m^-(\Gamma(\mathcal{A},0))+\sum_{t\in(0,1)}{\sgn\Gamma(\mathcal{A},t)}+m^+(\Gamma(\mathcal{A},1)).
\end{align}
\end{theorem}
It is worth to point out that these theorems provide the following method for computing the spectral flow of a general differentiable path $\mathcal{A}$ having invertible endpoints. Since $GL(W,H)$ is open, there exists $\delta_1>0$ such that $\mathcal{A}_0+\delta I_H$ and $\mathcal{A}_1+\delta I_H$ are invertible for all $0\leq\delta<\delta_1$. If we assume that $\delta_1$ is less than $\varepsilon$ in Theorem \ref{thm:Sard}, we conclude by the homotopy invariance of the spectral flow, that $\mathcal{A}^\delta:=\mathcal{A}+\delta I_H$ and $\mathcal{A}$ have the same spectral flow for all these $\delta$. By Theorem \ref{thm:Sard} there exists $0\leq \delta<\delta_1$ such that $\mathcal{A}^\delta$ has only regular crossings, and so we can use \eqref{sflcross} for computing the spectral flow of the original path $\mathcal{A}$. Note that $m^+(\Gamma(\mathcal{A}^\delta,0))$ and $m^-(\Gamma(\mathcal{A}^\delta,1))$ vanish in this case.\\
We now begin our proof of Theorem \ref{thm:Sard}. In what follows we denote by $M_1\sim M_2$ similarity of matrices $M_1,M_2$. Moreover, for quadratic forms $q_1,q_2$ defined on finite dimensional Hilbert spaces, we write $q_1\sim q_2$ if the representing matrices are similar (cf. Appendix A). As in \cite{Robbin-Salamon}, our argument is based on the following version of Kato's selection theorem (cf. \cite[Thm. II.5.4\& Thm. II.6.8]{Kato}). 

\begin{theorem}\label{thm:Kato}
Let $m\in\mathbb{N}$ and let $A:I\rightarrow M(m,\mathbb{R})$ be a $C^1$-path of selfadjoint matrices. Then there exists a $C^1$-path $\Delta$ of diagonal matrices

\[\Delta(t)=\diag(\lambda_1(t),\ldots,\lambda_m(t)),\quad t\in I,\]
such that $A_t\sim\Delta_t$, $t\in I$. Moreover,

\[\Gamma(\Delta-\delta,t)\sim\Gamma(A-\delta,t)\]
for all $t\in I$, $\delta\in\mathbb{R}$.
\end{theorem}

Note that the numbers $\lambda_j(t)$, $j=1,\ldots,m$, are the eigenvalues of the matrix $A_t$. However, let us point out that Theorem \ref{thm:Kato} does not assert that there is a continuous family of invertible matrices $Q$ such that $A_t=Q^{-1}_t\Delta_t Q_t$ for all $t\in I$, which is not true in general (cf. \cite[\S II.5.3]{Kato}).

\begin{lemma}\label{lemma:eigenvectors}
Let $t_0\in(0,1)$ and $a>0$ such that $\pm a\notin\sigma(\mathcal{A}_{t_0})$ and $\sigma(\mathcal{A}_{t_0})\cap[-a,a]$ consists only of eigenvalues of finite multiplicity. Then there exists $\varepsilon>0$ and $C^1$-functions 

\[\xi_1,\ldots,\xi_m:(t_0-\varepsilon,t_0+\varepsilon)\rightarrow W\]
such that $\{\xi_1(t),\ldots,\xi_m(t)\}$ is a basis of $E_{[-a,a]}(\mathcal{A}_t)$ for all $t\in(t_0-\varepsilon,t_0+\varepsilon)$.
\end{lemma}

\begin{proof} 
We first recall that by Corollary \ref{lem:sflneighbourhood}, there is $\varepsilon'>0$ such that $\pm a\notin\sigma(\mathcal{A}_{t})$ and $\sigma(\mathcal{A}_{t})\cap[-a,a]$ consists only of eigenvalues of finite multiplicity for all $t\in(t_0-\varepsilon',t_0+\varepsilon')$. In particular, $P_{[-a,a]}(\mathcal{A}_t)$ is defined for all $t\in(t_0-\varepsilon',t_0+\varepsilon')$, and we note that $P_{[-a,a]}(\mathcal{A}_t)$ defines a path in $\mathcal{L}(H,W)$ which actually depends $C^1$ on $t$. Indeed, if $\mathcal{O}\subset\mathbb{C}$ is an open set such that $\mathcal{O}\cap\sigma(\mathcal{A}_t)=\emptyset$ for all $t\in(t_0-\varepsilon',t_0+\varepsilon')$, then the map

\[\mathcal{O}\times(t_0-\varepsilon',t_0+\varepsilon')\ni(\lambda,t)\mapsto(\lambda-\mathcal{A}_t)^{-1}\in\mathcal{L}(H,W)\]
is $C^1$, and now differentiation under the integral sign shows the assertion (cf. \cite[Thm. II.5.4]{Kato}).\\
The functions $\xi_1,\ldots,\xi_m:(t_0-\varepsilon,t_0+\varepsilon)\rightarrow W$ can now be constructed as follows: the operator

\[B_t:=(I_H-P_{[-a,a]}(\mathcal{A}_{t_0}))+P_{[-a,a]}(\mathcal{A}_t)\in\mathcal{L}(H)\]
is bijective for $t=t_0$ and because $GL(H)\subset\mathcal{L}(H)$ is open, there exists $0<\varepsilon<\varepsilon'$ such that $B_t$ is an isomorphism for all $t\in(t_0-\varepsilon,t_0+\varepsilon)$. Consequently, since $B_t(\im(P_{[-a,a]}(\mathcal{A}_{t_0})))\subset \im(P_{[-a,a]}(\mathcal{A}_{t}))$ and both spaces are of the same dimension, we conclude that $P_{[-a,a]}(\mathcal{A}_t)$ maps  $\im(P_{[-a,a]}(\mathcal{A}_{t_0}))$ bijectively onto $\im(P_{[-a,a]}(\mathcal{A}_{t}))$ for $t\in(t_0-\varepsilon,t_0+\varepsilon)$. Now we take a basis $\xi^0,\ldots,\xi^m$ of $E_{[-a,a]}(\mathcal{A}_{t_0})$ and define $\xi_j(t)=P_{[-a,a]}(\mathcal{A}_t)\xi^j$ for $t\in(t_0-\varepsilon,t_0+\varepsilon)$ and $j=1,\ldots, m$. %Then $\xi_j:(t_0-\varepsilon,t_0+\varepsilon)\rightarrow W$ are $C^1$-functions such that $\xi_1(t),\ldots,\xi_n(t)$ is a basis of $E_{[-a,a]}(\mathcal{A}_{t})$, $t\in(t_0-\varepsilon,t_0+\varepsilon)$.  
\end{proof}

The following lemma recasts Kato's Selection Theorem \ref{thm:Kato} for operators in $\mathcal{FS}(W,H)$ and is essential for proving Theorem \ref{thm:Sard} and Theorem \ref{thm:crossing}.

\begin{lemma}\label{lemma:crossingform}
Let $t_0\in(0,1)$ and $a>0$ such that $\pm a\notin\sigma(\mathcal{A}_{t_0})$ and $\sigma(\mathcal{A}_{t_0})\cap[-a,a]$ consists only of eigenvalues of finite multiplicity. Then there exist $\varepsilon>0$ and a $C^1$-function $\Delta(t)$ of diagonal matrices such that

\[\Gamma(\Delta-\delta,t)\sim\Gamma(\mathcal{A}-\delta,t)\]
for $t_0-\varepsilon<t<t_0+\varepsilon$ and $-a<\delta<a$. 
\end{lemma}

\begin{proof}
We first recall that by Corollary \ref{lem:sflneighbourhood} and Lemma \ref{lemma:eigenvectors}, there exists $\varepsilon>0$ and $C^1$-functions $\xi_1,\ldots,\xi_m:(t_0-\varepsilon,t_0+\varepsilon)\rightarrow W$ such that $\pm a\notin\sigma(\mathcal{A}_{t})$, $\sigma(\mathcal{A}_{t})\cap[-a,a]$ consists only of eigenvalues of finite multiplicity and $\{\xi_1(t),\ldots,\xi_m(t)\}$ is a basis of $E_{[-a,a]}(\mathcal{A}_t)$ for all $t_0-\varepsilon<t<t_0+\varepsilon$. Clearly, by using a Gram-Schmidt process, we may assume that these bases are orthonormal in $H$.\\
We now define linear operators $\pi_t:H\rightarrow\mathbb{R}^m$ by

\[\pi^\ast_t u=\sum^m_{j=1}{u_j\xi_j(t)}\in E_{[-a,a]}(\mathcal{A}_t)\subset H,\]
and obtain a $C^1$-path of selfadjoint matrices

\[B_t:=\pi_t\mathcal{A}_t\pi^\ast_t,\quad t\in(t_0-\varepsilon,t_0+\varepsilon),\]
to which we may apply Kato's Selection Theorem \ref{thm:Kato}. Hence the task is now to show that $\Gamma(\mathcal{A}-\delta,t)\sim\Gamma(B-\delta,t)$ for all $t_0-\varepsilon<t<t_0+\varepsilon$ and $-a<\delta<a$.\\
We first observe that $\pi_t\pi^\ast_t=I_{\mathbb{R}^n}$ which easily follows from the orthonormality of $\{\xi_1(t),\ldots,\xi_m(t)\}$. Consequently, $\pi^\ast_t$ is injective and

\begin{align}\label{equ:crossing}
B_t-\delta=\pi_t(\mathcal{A}_t-\delta)\pi^\ast_t.
\end{align}    
Since $\pi_t$ is injective on $\im((\mathcal{A}_t-\delta)\pi^\ast_t)\subset E_{[-a,a]}(\mathcal{A}_t)$, we conclude from \eqref{equ:crossing} that $\pi^\ast_t$ induces an isomorphism 

\[\pi^\ast_t:\ker(B_t-\delta)\rightarrow\ker(\mathcal{A}_t-\delta).\]
Finally, we obtain for $u\in\ker(B_t-\delta)$

\begin{align*}
\Gamma(B-\delta,t)[u]&=\langle\dot{B}_tu,u\rangle_{\mathbb{R}^n}=\langle\pi_{t}\dot{\mathcal{A}}_t\pi^\ast_{t}u,u\rangle_{\mathbb{R}^n}\\
&+\langle\dot{\pi}_t(\mathcal{A}_{t}-\delta)\pi^\ast_{t} u,u\rangle_{\mathbb{R}^n}+\langle\pi_{t}(\mathcal{A}_{t}-\delta)\dot{\pi}^\ast_tu,u\rangle_{\mathbb{R}^n}\\
&=\langle\pi_{t}\dot{\mathcal{A}}_t\pi^\ast_{t}u,u\rangle_{\mathbb{R}^n}+\langle\dot{\pi}_t(\mathcal{A}_{t}-\delta )\pi^\ast_{t} u,u\rangle_{\mathbb{R}^n}+\langle\dot{\pi}^\ast_tu,(\mathcal{A}_{t}-\delta)\pi^\ast_{t}u\rangle_{H}\\
&=\langle\pi_{t}\dot{\mathcal{A}}_t\pi^\ast_{t}u,u\rangle_{\mathbb{R}^n}=\Gamma(\mathcal{A}-\delta,t)[\pi^\ast_{t}u],
\end{align*} 
which shows that $\Gamma(\Delta-\delta,t)\sim\Gamma(\mathcal{A}-\delta,t)$.
\end{proof}

It is worth to note that Lemma \ref{lemma:crossingform} implies that the diagonal entries $\lambda_j(t)$ of $\Delta(t)$ are the eigenvalues of $\mathcal{A}_t$ between $-a$ and $a$. Moreover, the derivatives $\dot{\lambda}_j(t)$ for those $j$ with $\lambda_j(t)=\delta$ are the eigenvalues of the crossing operator $\Gamma(\mathcal{A}-\delta,t)$.\\
We are now in the position to prove Theorem \ref{thm:Sard}. We choose as in the definition of the spectral flow a partition $0=t_0<t_1<\ldots<t_n$ of $I$ and $a_1,\ldots,a_n>0$ such that $\sigma(\mathcal{A}_t)\cap[-a_i,a_i]$ consists of eigenvalues of finite multiplicity and $\pm a_i\notin\sigma(\mathcal{A}_t)$ for all $t\in[t_{i-1},t_i]$, $i=1,\ldots,n$. By Lemma \ref{lemma:crossingform} we can cover the set

\[\bigcup^n_{i=1}{\{(t,\lambda)\in[t_{i-1},t_i]\times[-a_i,a_i]:\,\lambda\in\sigma(\mathcal{A}_t)\}}\]  
by finitely many graphs of $C^1$-functions $\lambda_j$, each defined on some subinterval $[s_{j-1},s_j]$ of one of the intervals $[t_{i-1},t_i]$.\\
Since the set of Fredholm operators is open in $\mathcal{L}(W,H)$, there is $0<\tilde{\delta}<\min_{i=1,\ldots n} a_i$ such that $\mathcal{A}_t+\tilde{\delta}$ is Fredholm for all $t\in I$. By Sard's theorem, the complement of the set of common regular values of the functions $\lambda_j$ in $[-\tilde{\delta},\tilde{\delta}]$ has measure zero. Finally, we see from Lemma \ref{lemma:crossingform} that $\delta\in[-\tilde{\delta},\tilde{\delta}]$ is a common regular value of the functions $\lambda_j$ if and only if $\mathcal{A}-\delta\,I_H$ has only regular crossings, and the proof of Theorem \ref{thm:Sard} is complete.\\
It remains to prove Theorem \ref{thm:crossing}. Let $t_0\in(0,1)$ be a crossing, set $m:=\dim\ker\mathcal{A}_{t_0}$ and take $a>0$ such that $\sigma(\mathcal{A}_{t_0})\cap[-a,a]=\{0\}$. Let $\varepsilon>0$ be as in Lemma \ref{lemma:crossingform} and

\[\lambda_1,\ldots,\lambda_m:(t_0-\varepsilon,t_0+\varepsilon)\rightarrow(-a,a)\]
corresponding $C^1$-functions representing the eigenvalues of $\mathcal{A}_t$ in $(-a,a)$. Since $t_0$ is a regular crossing of $\mathcal{A}$, we deduce that $\dot{\lambda}_j(t_0)\neq 0$, $j=1,\ldots,m$, and consequently, there exists $0<\eta<\frac{\varepsilon}{2}$ such that $\lambda_j(t)\neq 0$ for all $t\in(t_0-2\eta,t_0+2\eta)$, $t\neq t_0$, and $j=1,\ldots,m$. This in particular shows that regular crossings are isolated. Moreover,

\[\sgn \lambda_j(t_0+\eta)=-\sgn\lambda_j(t_0-\eta)=\sgn\dot{\lambda}_j(t_0),\,\,j=1,\ldots,m,\]  
and hence

\[\sgn\Gamma(\mathcal{A},t_0)=\dim E_{[0,a]}(\mathcal{A}_{t_0+\eta})-\dim E_{[0,a]}(\mathcal{A}_{t_0-\eta}).\]
We leave it to the reader to check that by a similar argument

\[-m^-(\Gamma(\mathcal{A},0))=\dim E_{[0,a]}(\mathcal{A}_{\eta})-\dim E_{[0,a]}(\mathcal{A}_{0})\]
and

\[m^+(\Gamma(\mathcal{A},1))=\dim E_{[0,a]}(\mathcal{A}_{1})-\dim E_{[0,a]}(\mathcal{A}_{1-\eta}),\]
where $a$ and $\eta$ are chosen similarly as before.\\
Now \eqref{sflcross} is a direct consequence of Lemma \ref{IndPre-lemma-sflbasicprop} i) and the definition of the spectral flow \eqref{IndPre-align-specflow}.

%%%%%%%%%%%%%%%%%%%%%%%%%%%%%%%%%%%%%%%%%%%%%%%%%%%%%%%%%%%%%%%%%%%%%%%%%%%%%%%%%%%%%%%%%%%%%%%%%%%%%%%%%%%%%%%%%%%%%%%%%%%%%%%%%%%%%%%%%%%%%%%%%%%%%%%%%%%%%%%%%%%%%%%%%%%%%%%%%%%%%%%%%%%%%%%%%%%%%%%%%%%%%%%%%%%%%%%%%%%%%%%%%%%%%%%%%%%%%%%%%%%%%%%%%%%%%%%%%%%%%%%%%%%%%%%%%%%%%%%%%%%%%%%%%%%%%%%%%%%%%%%%%%%%%%%%%%%%%%%%%%%%%%%%%%%%%%%%%%%%%%%%%%%%%%%%%%%%%%%%%%

\section{The Maslov index - a brief recapitulation}\label{sect:Maslov}
In this section we briefly recall the definition of the Maslov index for paths of Lagrangian subspaces of a symplectic space, where we follow \cite{Robbin-SalamonMAS}. In order to simplify the presentation, we only consider the symplectic space $\mathbb{R}^{2n}$ endowed with the standard scalar product $\langle \cdot,\cdot\rangle$ and the symplectic form $\omega(u,v)=\langle Ju,v\rangle$, $u,v\in\mathbb{R}^{2n}$, where $J$ denotes the matrix \eqref{J}.\\
At first, let us recall that the Grassmannian $G_n(\mathbb{R}^{2n})$ is a smooth $n^2$-dimensional manifold consisting of all $n$-dimensional subspaces of $\mathbb{R}^{2n}$. Its topology can be described by the metric

\begin{align*}
d(V,W)=\|P_V-P_W\|,
\end{align*}
where $P_V$ and $P_W$ denote the orthogonal projections in $\mathbb{R}^{2n}$ onto the subspaces $V$ and $W$, respectively.\\ 
A subspace $V\subset\mathbb{R}^{2n}$ is called Lagrangian if $\dim V=n$ and the restriction of the symplectic form $\omega$ to $V$ is trivial. The set $\Lambda(n)$ of all Lagrangian subspaces of $\mathbb{R}^{2n}$ is a $\frac{1}{2}n(n+1)$-dimensional smooth submanifold of $G_n(\mathbb{R}^{2n})$. If we fix some $V\in\Lambda(n)$, then 

\[\Sigma_k(V)=\{W\in\Lambda(n):\,\dim(V\cap W)=k\},\quad 0\leq k\leq n,\]
are connected $\frac{1}{2}k(k+1)$-codimensional submanifolds of $\Lambda(n)$ and

\[\Lambda(n)=\bigcup^n_{k=0}{\Sigma_k(V)},\quad\Sigma(V):=\bigcup^n_{k=1}{\Sigma_k(V)}=\overline{\Sigma_1(V)}.\]
The latter set is called the Maslov cycle of $V$.\\
We now consider a smooth path $\gamma:I\rightarrow\Lambda(n)$ and assume that the endpoints of $\gamma$ do not belong to $\Sigma(V)$. We say that $\lambda_0\in I$ is a crossing of $\gamma$ if $\gamma(\lambda_0)\in\Sigma(V)$, i.e. $\gamma(\lambda_0)\cap V\neq\{0\}$. If $W\in\Lambda(n)$ is transversal to $\gamma(\lambda_0)$, then there exists a continuously differentiable family of linear maps $\varphi_\lambda:\gamma(\lambda_0)\rightarrow W$, such that every element in $\gamma(\lambda)\subset\mathbb{R}^{2n}$ for $|\lambda-\lambda_0|$ sufficiently small can be uniquely written as $v+\varphi_\lambda(v)$ for some $v\in\gamma(\lambda_0)$. The \textit{crossing form} at a crossing $\lambda_0$ is the quadratic form defined by

\[\Gamma(\gamma,V,\lambda_0):\gamma(\lambda_0)\cap V\rightarrow\mathbb{R},\quad \Gamma(\gamma,V,\lambda_0)[v]=\frac{d}{d\lambda}\mid_{\lambda=\lambda_0}\omega(v,\varphi_\lambda(v)).\]
A crossing $\lambda_0$ is called \textit{regular} if $\Gamma(\gamma,V,\lambda_0)$ is non-degenerate, and heuristically, $\gamma$ has only regular crossings if and only if it is transverse to $\Sigma(V)$. Since regular crossings are isolated, the Maslov index of $\gamma$ with respect to $V$ can be defined in this case by

\begin{align}\label{Maslov}
\mu_{Mas}(\gamma,V)=\sum_{0<\lambda<1}{\sgn\Gamma(\gamma,V,\lambda).}
\end{align}
It is shown in \cite[\S 2]{Robbin-SalamonMAS}, that this definition extends to an integer valued homotopy invariant on the set of all paths in $\Lambda(n)$ having endpoints in $\Lambda(n)\setminus\Sigma(V)$.\\
For later reference, we note an important special case in which the crossing forms can be computed explicitly (cf. \cite[Rem. 5.34]{Robbin-Salamon}). Let us assume that $\Psi:I\rightarrow\Sp(2n)$ is a differentiable path of symplectic matrices, $V=\{0\}\times\mathbb{R}^n\in\Lambda(n)$ and let us denote by $\Psi_\cdot V$ the path $I\ni\lambda\mapsto\Psi_\lambda V\in\Lambda(n)$. If we write

\begin{align*}
\Psi_\lambda=\begin{pmatrix}
a_\lambda&b_\lambda\\
c_\lambda&d_\lambda
\end{pmatrix}\in\Sp(2n),\quad \lambda\in I,
\end{align*}
then

\[\Psi_\lambda(\{0\}\times\mathbb{R}^n)\cap(\{0\}\times\mathbb{R}^n)=\{(0,d_\lambda u)\in\mathbb{R}^n\times\mathbb{R}^n:\,u\in\ker b_\lambda\},\]
and in particular, $\lambda_0\in I$ is a crossing if and only if $\ker b_{\lambda_0}\neq\{0\}$. The crossing form at a crossing $\lambda_0$ turns out to be

\[\Gamma(\Psi_\cdot V,V,\lambda_0)[v]=-\langle d_{\lambda_0}u,\dot{b}_{\lambda_0}u\rangle,\]
where $v=(0,d_{\lambda_0}u)$ and $b_{\lambda_0}u=0$. We now introduce a quadratic form 

\begin{align}\label{MaslovSympMat}
q:\ker b_{\lambda_0}\rightarrow\mathbb{R},\quad q[u]=\Gamma(\Psi_\cdot V,V,\lambda_0)[(0,d_{\lambda_0}u)]=-\langle d_{\lambda_0}u,\dot{b}_{\lambda_0}u\rangle,
\end{align}	
and since 

\[d_{\lambda_0}\mid_{\ker b_{\lambda_0}}:\ker b_{\lambda_0}\rightarrow \Psi_{\lambda_0}(\{0\}\times\mathbb{R}^n)\cap(\{0\}\times\mathbb{R}^n)\]
is an isomorphism, we conclude that $\lambda_0$ is a regular crossing if and only if $q$ is non-degenerate. Moreover, the contribution of $\lambda_0$ to the Maslov index of $\Psi_\cdot V$ in \eqref{Maslov} is given by the signature of $q$.\\
Finally, let us introduce the relative Maslov index for pairs of paths $(\gamma_1,\gamma_2):I\rightarrow\Lambda(n)\times\Lambda(n)$. The space $\mathbb{R}^{2n}\times\mathbb{R}^{2n}$ is symplectic with respect to the symplectic form induced by $J\times(-J)$. Clearly, every element in $\Lambda(n)\times\Lambda(n)$ belongs to $\Lambda(2n)$. Moreover, the diagonal $\Delta\subset\mathbb{R}^{2n}\times\mathbb{R}^{2n}$ is Lagrangian, and $(\gamma_1(\lambda),\gamma_2(\lambda))\cap\Delta\neq\{0\}$ if and only if $\gamma_1(\lambda)\cap\gamma_2(\lambda)\neq\{0\}$. We define the relative Maslov index by

\[\mu_{Mas}(\gamma_1,\gamma_2)=\mu_{Mas}(\gamma_1\times\gamma_2,\Delta),\]
where we assume that $\gamma_1(0)\cap\gamma_2(0)=\gamma_1(1)\cap\gamma_2(1)=\{0\}$, and we note the following two properties:

\begin{enumerate}
	\item[i)] If $h=(h_1,h_2):I\times I\rightarrow\Lambda(n)\times\Lambda(n)$ is such that $h_1(0,\tau)\cap h_2(0,\tau)=\{0\}$ and $h_1(1,\tau)\cap h_2(1,\tau)=\{0\}$ for all $\tau\in I$, then $\mu_{Mas}(h(\cdot,0))=\mu_{Mas}(h(\cdot,1))$.
	\item[ii)] If $\Psi:I\rightarrow\Sp(2n)$ is a path of symplectic matrices, then
	
	\begin{align}\label{Maslovnatural}
	\mu_{Mas}(\Psi\,\gamma_1,\Psi\,\gamma_2)=\mu_{Mas}(\gamma_1,\gamma_2),
	\end{align} 
	where $\Psi\gamma_i:I\rightarrow\Lambda(n)$ is the path $(\Psi\gamma_i)(\lambda)=\Psi_\lambda\gamma_i(\lambda)$, $i=1,2$.
\end{enumerate}
We call $\lambda_0\in I$ a crossing if $\gamma_1(\lambda_0)\cap\gamma_2(\lambda_0)\neq \{0\}$ and we define the relative crossing form at a crossing $\lambda_0$ by

\[\Gamma(\gamma_1,\gamma_2,\lambda_0)=\Gamma(\gamma_1,\gamma_2(\lambda_0),\lambda_0)-\Gamma(\gamma_2,\gamma_1(\lambda_0),\lambda_0).\]
As before, regular crossings are isolated and if $(\gamma_1,\gamma_2)$ has only regular crossings, then 

\begin{align}\label{regMas}
\mu_{Mas}(\gamma_1,\gamma_2)=\sum_{0<\lambda<1}{\sgn\Gamma(\gamma_1,\gamma_2,\lambda)}.
\end{align}

%%%%%%%%%%%%%%%%%%%%%%%%%%%%%%%%%%%%%%%%%%%%%%%%%%%%%%%%%%%%%%%%%%%%%%%%%%%%%%%%%%%%%%%%%%%%%%%%%%%%%%%%%%%%%%%%%%%%%%%%%%%%%%%%%%%%%%%%%%%%%%%%%%%%%%%%%%%%%%%%%%%%%%%%%%%%%%%%%%%%%%%%%%%%%%%%%%%%%%%%%%%%%%%%%%%%%%%%%%%%%%%%%%%%%%%%%%%%%%%%%%%%%%%%%%%%%%%%%%%%%%%%%%%%%%%%%%%%%%%%%%%%%%%%%%%%%%%%%%%%%%%%%%%%%%%%%%%%%%%%%%%%%%%%%%%%%%%%%%%%%%%%%%%%%%%%%%%%%%%%%%

\section{The main theorem}

\subsection{Assumptions and statement of the theorem}\label{sect:main}
Let $S:I\times\mathbb{R}\rightarrow\mathcal{L}(\mathbb{R}^{2n})$ be a smooth family of symmetric matrices such that $S_\lambda:=S(\lambda,\cdot):\mathbb{R}\rightarrow\mathcal{L}(\mathbb{R}^{2n})$ converges uniformly in $\lambda$ to families 

\[S_\lambda(\infty)=\lim_{t\rightarrow\infty}S_\lambda(t),\quad S_\lambda(-\infty)=\lim_{t\rightarrow-\infty}S_\lambda(t),\quad\lambda\in I.\]
Let us also assume that for $\dot{S}_\lambda$, the derivative with respect to $\lambda$, there is a constant $C_1>0$ such that
  
\begin{align}\label{growth}
\|\dot{S}_\lambda(t)\|<C_1,\quad (\lambda,t)\in I\times\mathbb{R}.
\end{align}
We consider the family of Hamiltonian systems \eqref{Hamiltonian} from the introduction, i.e.

\begin{equation*}
\left\{
\begin{aligned}
Ju'(t)+S_\lambda(t)u(t)&=0,\quad t\in\mathbb{R}\\
\lim_{t\rightarrow\pm\infty}u(t)&=0,
\end{aligned}
\right.
\end{equation*}
where $J$ denotes the symplectic matrix \eqref{J}. If $u:\mathbb{R}\rightarrow\mathbb{R}^{2n}$ is a solution of \eqref{Hamiltonian} for some $\lambda\in I$, then clearly $u$ is smooth. Henceforth we assume

\begin{enumerate}
	\item[A1)] The matrices $JS_\lambda(\pm\infty)$ are hyperbolic, i.e. they have no eigenvalues on the imaginary axis.
\end{enumerate}
Then there exist constants $\alpha,\beta>0$ such that

\begin{align}\label{decay}
|u(t)|\leq \alpha e^{-\beta|t|},\quad t\in\mathbb{R},
\end{align}
and consequently, $u$ and so $u'=JS_\lambda u$ are elements of $L^2(\mathbb{R},\mathbb{R}^{2n})$ (cf. \cite[Lemma 1.1]{Alberto}). Let us recall that the space $H^1(\mathbb{R},\mathbb{R}^{2n})$ of all absolutely continuous functions in $L^2(\mathbb{R},\mathbb{R}^{2n})$ having derivatives in $L^2(\mathbb{R},\mathbb{R}^{2n})$ is a Hilbert space with respect to the scalar product

\[\langle u,v\rangle_{H^1(\mathbb{R},\mathbb{R}^{2n})}=\langle u,v\rangle_{L^2(\mathbb{R},\mathbb{R}^{2n})}+\langle u',v'\rangle_{L^2(\mathbb{R},\mathbb{R}^{2n})},\quad u,v\in H^1(\mathbb{R},\mathbb{R}^{2n}).\]
In particular, every solution of \eqref{Hamiltonian} belongs to $H^1(\mathbb{R},\mathbb{R}^{2n})$, and we now define operators

\begin{align}\label{diffop}
(\mathcal{A}_\lambda u)(t):=Ju'(t)+S_\lambda(t)u(t),\quad \lambda\in I,
\end{align}
acting between the spaces

\[W=H^1(\mathbb{R},\mathbb{R}^{2n})\quad\text{and}\quad H=L^2(\mathbb{R},\mathbb{R}^{2n}).\]
We see at once that $\mathcal{A}$ is a continuously differentiable family in $\mathcal{L}(W,H)$, and moreover, it is an easy exercise to show that $\mathcal{A}_\lambda\in\mathcal{S}(W,H)$, $\lambda\in I$. It is well known (cf. \cite[Thm. 2.1]{Robbin-Salamon}) that under the assumption A 1) each $\mathcal{A}_\lambda$ is Fredholm. Finally, one additional assumption is necessary in order to obtain a path in $\mathcal{FS}(W,H)$ having invertible ends:

\begin{enumerate}
	\item[A2)] The equations \eqref{Hamiltonian} admit only the trivial solution for $\lambda=0$ and $\lambda=1$. 
\end{enumerate}
Since selfadjoint Fredholm operators have a vanishing Fredholm index, A2) holds if and only if $\mathcal{A}_0$ and $\mathcal{A}_1$ are invertible. Hence if A1) and A2) are satisfied, then $\mathcal{A}$ is a path in $\mathcal{FS}(W,H)$ having invertible ends, and so its spectral flow is defined according to Section \ref{sect:sfl}.\\
Let us now assign a second integer to the equations \eqref{Hamiltonian} by using the Maslov index from Section \ref{sect:Maslov}. We define a two parameter family of matrix-valued maps $\Psi_{(\lambda,t_0)}:\mathbb{R}\rightarrow\mathcal{L}(\mathbb{R}^{2n})$, $(\lambda,t_0)\in I\times\mathbb{R}$, by

\begin{equation*}
\left\{
\begin{aligned}
J\Psi'_{(\lambda,t_0)}(t)+S_\lambda(t)\Psi_{(\lambda,t_0)}(t)&=0,\quad t\in\mathbb{R}\\
\Psi_{(\lambda,t_0)}(t_0)&=I_{2n},
\end{aligned}
\right.
\end{equation*}
and we note that the stable and unstable subspaces of \eqref{Hamiltonian} at $t_0\in\mathbb{R}$ can be written as

\begin{align*}
E^s_\lambda(t_0)&=\{v\in\mathbb{R}^{2n}:\,\Psi_{(\lambda,t_0)}(t)v\rightarrow 0,\,\, t\rightarrow\infty\}\\
E^u_\lambda(t_0)&=\{v\in\mathbb{R}^{2n}:\,\Psi_{(\lambda,t_0)}(t)v\rightarrow 0,\,\, t\rightarrow -\infty\}.
\end{align*}
Moreover, we define

\begin{align}\label{stabinf}
\begin{split}
E^s_\lambda(\pm\infty)&:=\{u(0):\, Ju'(t)+S_\lambda(\pm\infty) u(t)=0,\, t\in\mathbb{R};\,u(t)\rightarrow 0,t\rightarrow\infty \}\\
&=\{v\in\mathbb{R}^{2n}:\,\lim_{t\rightarrow\infty} e^{tJS_\lambda(\pm\infty)}v=0\}
\end{split}
\end{align}
and

\begin{align}\label{unstabinf}
\begin{split}
E^u_\lambda(\pm\infty)&:=\{u(0):\, Ju'(t)+S_\lambda(\pm\infty) u(t)=0,\, t\in\mathbb{R};\,u(t)\rightarrow 0,t\rightarrow-\infty \}\\
&=\{v\in\mathbb{R}^{2n}:\,\lim_{t\rightarrow-\infty} e^{tJS_\lambda(\pm\infty)}v=0\}.
\end{split}
\end{align}
Among the many references about stable and unstable subspaces, we particularly want to mention the beautiful paper \cite{Alberto} which even treats the more general infinite dimensional theory. Here we will not need to know much about these spaces. Essentially, we only require that under the given assumptions, $E^s_\lambda(t)$ and $E^u_\lambda(t)$ define smooth families in the Grassmannian $G(\mathbb{R}^{2n})$, and 

\begin{align}\label{stableunstablelimits}
\lim_{t\rightarrow\infty}E^s_\lambda(t)=E^s_\lambda(+\infty),\quad \lim_{t\rightarrow-\infty}E^u_\lambda(t)=E^u_\lambda(-\infty),\quad\lambda\in I.
\end{align}
Finally, let us note for later reference the elementary fact that for each $\lambda\in I$ and $t_0\in\mathbb{R}$, the evaluation map 

\begin{align}\label{evaliso}
\ker\mathcal{A}_\lambda\rightarrow E^u_\lambda(t_0)\cap E^s_\lambda(t_0),\quad u\mapsto u(t_0)
\end{align}
is an isomorphism.

\begin{lemma}\label{stabsympl}
The spaces $E^s_\lambda(t)$, $(\lambda,t)\in I\times[0,\infty]$, and $E^u_\lambda(t)$, $(\lambda,t)\in I\times[-\infty,0]$, belong to $\Lambda(n)$.
\end{lemma}

\begin{proof}
Let us first recall that the dimension of an isotropic subspace of $\mathbb{R}^{2n}$ is at most $n$, and that an isotropic subspace is Lagrangian if and only if its dimension is equal to $n$.\\
Now, if $v,w:\mathbb{R}\rightarrow\mathbb{R}^{2n}$ are solutions of the differential equation $Ju'+S_\lambda u=0$, then $\omega(v(t),w(t))$ is constant for all $t\in\mathbb{R}$. This clearly implies that $\omega(v(t_0),w(t_0))=0$ if either $v(t_0),w(t_0)\in E^u_\lambda(t_0)$ or $v(t_0),w(t_0)\in E^s_\lambda(t_0)$ for some $t_0\in\mathbb{R}$. Since the same argument applies to the equations $Ju'+S_\lambda(\pm\infty)u=0$, we conclude that the spaces $E^u_\lambda(t_0)$, $E^s_\lambda(t_0)$, $E^u_\lambda(\pm\infty)$ and $E^s_\lambda(\pm\infty)$ are isotropic.\\ 
Note that $E^s_\lambda(\pm\infty)$ are the generalised eigenspaces of $JS_\lambda(\pm\infty)$ with respect to eigenvalues having negative real part, and $E^u_\lambda(\pm\infty)$ are the generalised eigenspaces of $JS_\lambda(\pm\infty)$ with respect to eigenvalues having positive real part (cf. \cite[\S 12]{Amann}). Since $JS_\lambda(\pm\infty)$ are hyperbolic by A1), we conclude that $E^s_\lambda(+\infty)\oplus E^u_\lambda(+\infty)=\mathbb{R}^{2n}=E^s_\lambda(-\infty)\oplus E^u_\lambda(-\infty)$. It follows that all these spaces are of dimension $n$ and thus Lagrangian.\\
Finally, we deduce from \eqref{stableunstablelimits} that $\dim E^s_\lambda(t)=\dim E^u_\lambda(t)=n$, $t\in\mathbb{R}$, which shows that these spaces are Lagrangian as well. 
\end{proof}

In what follows we denote for $t_0\in\mathbb{R}$ by $E^u_{\cdot}(t_0)$ and $E^s_{\cdot}(t_0)$ the paths $E^u_{\lambda}(t_0),E^s_{\lambda}(t_0)$ in $\Lambda(n)$ parametrised by $\lambda\in I$. From our assumption A2), the differential equations \eqref{Hamiltonian} only have the trivial solution for $\lambda=0,1$. Accordingly, $E^u_0(t_0)\cap E^s_0(t_0)=E^u_1(t_0)\cap E^s_1(t_0)=\{0\}$ for all $t_0\in\mathbb{R}$ by \eqref{evaliso}, and so the relative Maslov index $\mu_{Mas}(E^u_{\cdot}(t_0),E^s_{\cdot}(t_0))$ is defined.

\begin{lemma}\label{Maslovconstant}
$\mu_{Mas}(E^u_{\cdot}(t_0),E^s_{\cdot}(t_0))=\mu_{Mas}(E^u_{\cdot}(0),E^s_{\cdot}(0))$ for all $t_0\in\mathbb{R}$.
\end{lemma}

\begin{proof}
We obtain from \eqref{evaliso} that $E^u_\lambda(t_0)\cap E^s_\lambda(t_0)=\{0\}$ for all $t_0\in\mathbb{R}$ and $\lambda=0,1$. Consequently,

\[h=(h_1,h_2):I\times I\rightarrow\Lambda(n)\times\Lambda(n),\quad h(\lambda,\tau)=(E^u_\lambda((1-\tau)\cdot t_0),E^s_\lambda((1-\tau)\cdot t_0))\]
defines a homotopy such that $h_1(0,\tau)\cap h_2(0,\tau)=\{0\}$ and $h_1(1,\tau)\cap h_2(1,\tau)=\{0\}$. The assertion follows by the homotopy invariance of the Maslov index.
\end{proof}

Finally we can state our main theorem.

\begin{theorem}\label{main}
If the assumptions A1) and A2) introduced above hold for the family of equations \eqref{Hamiltonian}, then

\[\sfl(\mathcal{A})=\mu_{Mas}(E^u_\cdot(0),E^s_\cdot(0)).\] 
\end{theorem}

%%%%%%%%%%%%%%%%%%%%%%%%%%%%%%%%%%%%%%%%%%%%%%%%%%%%%%%%%%%%%%%%%%%%%%%%%%%%%%%%%%%%%%%%%%%%%%%%%%%%%%%%%%%%%%%%%%%%%%%%%%%%%%%%%%%%%%%%%%%%%%%%%%%%%%%%%%%%%%%%%%%%%%%%%%%%%%%%%%%%%%%%%%%%%%%%%%%%%%%%%%%%%%%%%%%%%%%%%%%%%%%%%%%%%%%%%%%%%%%%%%%%%%%%%%%%%%%%%%%%%%%%%%%%%%%%%%%%%%%%%%%%%%%%%%%%%%%%%%%%%%%%%%%%%%%%%%%%%%%%%%%%%%%%%%%%%%%%%%%%%%%%%%%%%%%%%%%%%%%%%%

\subsection{Proof of Theorem \ref{main}}
We divide the proof into eight steps.

%%%%%%%%%%%%%%%%%%%%%%%%%%%%%%%%%%%%%%%%%%%%%%%%%%%%%%%%%%%%%%%%%%%%%%%%%%%%%%%%%%%%%%%%%%%%%%%%%%%%%%%%%%%%%%%%%%%%%%%%%%%%%%%%%%%%%%%%%%%%%%%%%%%%%%%%%%%%%%%%%%%%%%%%%%%%%%%%%%%%%%%%%%%%%%%%%%%%%%%%%%%%%%%%%%%%%%%%%%%%%%%%%%%%%%%%%%%%%%%%%%%%%%%%%%%%%%%%%%%%%%%%%%%%%%%%%%%%%%%%%%%%%%%%%%%%%%%%%%%%%%%%%%%%%%%%%%%%%%%%%%%%%%%%%%%%%%%%%%%%%%%%%%%%%%%%%%%%%%%%%%

\subsubsection*{Step 1: A simple deformation}
Since by assumption the matrices $S_\lambda(t)$ converge for $t\rightarrow\pm\infty$ uniformly in $\lambda$, we can find for every $\varepsilon>0$ a smooth family $\hat{S}:I\times\mathbb{R}\rightarrow\mathcal{S}(\mathbb{R}^{2n})$ of symmetric matrices and $T>0$ such that 

\[\|S_\lambda(t)-\hat{S}_\lambda(t)\|<\varepsilon,\quad (\lambda,t)\in I\times\mathbb{R}\]
and $\hat{S}_\lambda(t)$ is (locally) constant in $t$ for all $|t|>T$.\\
Let us now consider for $\tau\in I$ the equations

\begin{equation}\label{equationtau}
\left\{
\begin{aligned}
Ju'(t)+((1-\tau)S_\lambda(t)+\tau\hat{S}_\lambda(t))u(t)&=0,\quad t\in\mathbb{R}\\
\lim_{t\rightarrow\pm\infty}u(t)&=0
\end{aligned}
\right.
\end{equation}
and let us denote by $h(\lambda,\tau)$ the corresponding differential operators mapping $W=H^1(\mathbb{R},\mathbb{R}^{2n})$ to $H=L^2(\mathbb{R},\mathbb{R}^{2n})$. Since $\mathcal{FS}(W,H)\subset\mathcal{S}(W,H)$ and $GL(W,H)\subset\mathcal{L}(W,H)$ are open, we can take $\varepsilon>0$ above sufficiently small, such that $h(I\times I)\subset\mathcal{FS}(W,H)$ and $h(0,\tau),h(1,\tau)$ are invertible for all $\tau\in I$. By the homotopy invariance of the spectral flow, we conclude that 

\[\sfl(\mathcal{A})=\sfl(h(\cdot,0))=\sfl(h(\cdot,1)).\]
Let us now denote by $E^s_\lambda(0,\tau)$ the stable subspaces and by $E^u_\lambda(0,\tau)$ the unstable subspaces of the equations \eqref{equationtau}, and consider the homotopy

\[\hat{h}=(\hat{h}_1,\hat{h}_2):I\times I\rightarrow\Lambda(n)\times\Lambda(n),\quad \hat{h}(\lambda,\tau)=(E^s_\lambda(0,\tau),E^u_\lambda(0,\tau)).\]
Since $\ker h(0,\tau)=\ker h(1,\tau)=\{0\}$, we see by \eqref{evaliso} that $\hat{h}_1(0,\tau)\cap \hat{h}_2(0,\tau)=\{0\}$ and $\hat{h}_1(1,\tau)\cap\hat{h}_2(1,\tau)=\{0\}$ for all $\tau\in I$, and finally we obtain from the homotopy invariance of the Maslov index 

\[\mu_{Mas}(E^s_\cdot(0,1),E^u_\cdot(0,1))=\mu_{Mas}(E^s_\cdot(0,0),E^u_\cdot(0,0))=\mu_{Mas}(E^s_\cdot(0),E^u_\cdot(0)).\] 
Consequently, we can henceforth assume without loss of generality that there exists $T>0$ such that $S_\lambda(t)=S_\lambda(T)$ for all $t\geq T$ and $S_\lambda(t)=S_\lambda(-T)$ for all $t\leq -T$, $\lambda\in I$. This particularly implies that

\begin{align}\label{constantstab}
\begin{split}
E^s_\lambda(t_0)&=\{v\in\mathbb{R}^{2n}:\lim_{t\rightarrow\infty} e^{(t-t_0)JS_\lambda(T)}v=0\}=\{v\in\mathbb{R}^{2n}:\lim_{t\rightarrow\infty} e^{(t-T)JS_\lambda(T)}v=0\}\\
&=E^s_\lambda(T),\quad t_0\geq T
\end{split}
\end{align}
and

\begin{align}\label{constantunstab}
\begin{split}
E^u_\lambda(t_0)&=\{v\in\mathbb{R}^{2n}:\lim_{t\rightarrow-\infty} e^{(t-t_0)JS_\lambda(-T)}v=0\}=\{v\in\mathbb{R}^{2n}:\lim_{t\rightarrow-\infty} e^{(t+T)JS_\lambda(-T)}v=0\}\\
&=E^u_\lambda(-T),\quad t_0\leq-T
\end{split}
\end{align} 
We now use Lemma \ref{famlifting} to find smooth maps $\phi:I\times I\times[0,\infty)\rightarrow\Sp(2n)$ and $\ell_1,\ell_2:[0,\infty)\rightarrow\Lambda(n)$ such that $\phi(0,\lambda,t)\ell_1(t)=E^u_\lambda(-t)$ and $\phi(1,\lambda,t)\ell_2(t)=E^s_\lambda(t)$. By \eqref{constantstab} and \eqref{constantunstab}, we can assume that $\phi(s,\lambda,t)$ is constant in $t$ if $t>T$. Hence there is a constant $C_2>0$ such that

\begin{align}\label{estimatespsi}
\frac{1}{C_2}\|\xi\|\leq\|\phi(s,\lambda,t)\xi\|\leq C_2\|\xi\|,\qquad \| \dot{\phi}(s,\lambda,t)\xi\|\leq C_2\|\xi\|
\end{align}
for all $\xi\in\mathbb{R}^{2n}$ and $(s,\lambda,t)\in I\times I\times[0,\infty)$, where $\dot{\phi}$ denotes the derivative with respect to $\lambda$.

%%%%%%%%%%%%%%%%%%%%%%%%%%%%%%%%%%%%%%%%%%%%%%%%%%%%%%%%%%%%%%%%%%%%%%%%%%%%%%%%%%%%%%%%%%%%%%%%%%%%%%%%%%%%%%%%%%%%%%%%%%%%%%%%%%%%%%%%%%%%%%%%%%%%%%%%%%%%%%%%%%%%%%%%%%%%%%%%%%%%%%%%%%%%%%%%%%%%%%%%%%%%%%%%%%%%%%%%%%%%%%%%%%%%%%%%%%%%%%%%%%%%%%%%%%%%%%%%%%%%%%%%%%%%%%%%%%%%%%%%%%%%%%%%%%%%%%%%%%%%%%%%%%%%%%%%%%%%%%%%%%%%%%%%%%%%%%%%%%%%%%%%%%%%%%%%%%%%%%%%%%

\subsubsection*{Step 2: Perturbations and regular crossings}
We consider the path $\mathcal{A}$ of selfadjoint Fredholm operators introduced in \eqref{diffop} and we set $\mathcal{A}^\delta:=\mathcal{A}+\delta I_H$ as in Section \ref{sect:sfl}. Since the operators $\mathcal{A}_0$ and $\mathcal{A}_1$ are invertible by assumption A2), there exists $\delta_1>0$ such that $\mathcal{A}^\delta_0, \mathcal{A}^\delta_1$ are invertible for all $|\delta|<\delta_1$. Moreover, by the first part of Theorem \ref{thm:Sard}, we can assume that $\mathcal{A}^\delta_t\in\mathcal{FS}(W,H)$ for all $t\in I$ and all $|\delta|<\delta_1$. Finally, we use the second part of Theorem \ref{thm:Sard} to find $0<\delta<\delta_1$ such that $\mathcal{A}^\delta$ has only regular crossings, and we note that the straight homotopy

\[h:I\times I\rightarrow\mathcal{FS}(W,H),\quad h(\lambda,\tau)=\mathcal{A}_\lambda+\tau\cdot\delta I_H\]
shows that $\sfl(\mathcal{A})=\sfl(\mathcal{A}^\delta)$.\\
Let us now consider the equations

\begin{equation}\label{HamiltonianII}
\left\{
\begin{aligned}
Ju'(t)+S_\lambda(t)u(t)+\delta\,u(t)&=0,\quad t\in\mathbb{R}\\
\lim_{t\rightarrow\pm\infty}u(t)&=0
\end{aligned}
\right.
\end{equation}
which correspond to the operators $\mathcal{A}^\delta$. We denote by $E^s_\lambda(0,\delta)$ the stable subspaces and by $E^u_\lambda(0,\delta)$ the unstable subspaces of the perturbed equations \eqref{HamiltonianII}, and we consider the homotopy

\[h=(h_1,h_2):I\times I\rightarrow\Lambda(n)\times\Lambda(n),\quad h(\lambda,\tau)=(E^s_\lambda(0,\tau\cdot\delta),E^u_\lambda(0,\tau\cdot\delta)).\]
Since $h_1(0,\tau)\cap h_2(0,\tau)=\{0\}$ and $h_1(1,\tau)\cap h_2(1,\tau)=\{0\}$ for all $\tau\in I$ by \eqref{evaliso}, we infer that 

\[\mu_{Mas}(E^s_\cdot(0,\delta),E^u_\cdot(0,\delta))=\mu_{Mas}(E^s_\cdot(0,0),E^u_\cdot(0,0))=\mu_{Mas}(E^s_\cdot(0),E^u_\cdot(0)).\] 
The conclusion of this second step is that it suffices to prove Theorem \ref{main} under the assumption that $\mathcal{A}$ has only regular crossings. Moreover, by Theorem \ref{thm:crossing} we can assume that there is only one crossing, which we henceforth denote by $\lambda_0$.\\

\vspace*{0.5cm}

For the rest of the proof, we fix a basis $\{u_1,\ldots,u_m\}$ of $\ker\mathcal{A}_{\lambda_0}$, which is orthonormal with respect to the scalar product of $H=L^2(\mathbb{R},\mathbb{R}^n)$. Let us note the following estimate for later reference.

\begin{lemma}\label{lem:estimatevalueu}
If $u\in\ker\mathcal{A}_{\lambda_0}$, then

\[\|u(t)\|\leq m\max_{i=1,\ldots m}\|u_i(t)\|\,\|u\|_H,\quad t\in\mathbb{R}.\]
\end{lemma}

\begin{proof}
Let $\alpha_i\in\mathbb{R}$ be such that $u=\sum^m_{i=1}{\alpha_iu_i}$. Since $u_1,\ldots,u_m$ is orthonormal in $H$, this in particular implies that $\sum^m_{i=1}\alpha^2_i=\|u\|^2_H$ and we obtain for $t\in\mathbb{R}$

\begin{align*}
\|u(t)\|^2&=\left\|\sum^m_{i=1}{\alpha_iu_i(t)}\right\|^2=\sum^m_{i,j=1}{\alpha_i\alpha_j\langle u_i(t),u_j(t)\rangle}\leq \sum^m_{i,j=1}{|\alpha_i||\alpha_j|\|u_i(t)\|\|u_j(t)\|}\\
&\leq\max_{i=1,\ldots,m}\|u_i(t)\|^2\sum^m_{i,j=1}{|\alpha_i||\alpha_j|}\leq m^2\max_{i=1,\ldots,m}\|u_i(t)\|^2\|u\|^2_H.
\end{align*}
\end{proof}

%%%%%%%%%%%%%%%%%%%%%%%%%%%%%%%%%%%%%%%%%%%%%%%%%%%%%%%%%%%%%%%%%%%%%%%%%%%%%%%%%%%%%%%%%%%%%%%%%%%%%%%%%%%%%%%%%%%%%%%%%%%%%%%%%%%%%%%%%%%%%%%%%%%%%%%%%%%%%%%%%%%%%%%%%%%%%%%%%%%%%%%%%%%%%%%%%%%%%%%%%%%%%%%%%%%%%%%%%%%%%%%%%%%%%%%%%%%%%%%%%%%%%%%%%%%%%%%%%%%%%%%%%%%%%%%%%%%%%%%%%%%%%%%%%%%%%%%%%%%%%%%%%%%%%%%%%%%%%%%%%%%%%%%%%%%%%%%%%%%%%%%%%%%%%%%%%%%%%%%%%%

\subsubsection*{Step 3: Restriction to a finite time interval - the quadratic forms $Q_1$}

By Theorem \ref{thm:crossing}, we have

\[\sfl(\mathcal{A})=\sgn\Gamma(\mathcal{A},\lambda_0),\]
where $\Gamma(\mathcal{A},\lambda_0)$ is the non-degenerate quadratic form on $\ker\mathcal{A}_{\lambda_0}$ defined by

\[\Gamma(\mathcal{A},\lambda_0)[u]=\langle\dot S_{\lambda_0}u,u\rangle_{L^2(\mathbb{R},\mathbb{R}^n)}=\int^\infty_{-\infty}{\langle\dot S_{\lambda_0}(t)u(t),u(t)\rangle\, dt},\quad u\in\ker\mathcal{A}_{\lambda_0}.\]
Our task in this step is to find a compact subinterval of $\mathbb{R}$ to which we may restrict the integration in the definition of $\Gamma(\mathcal{A},\lambda_0)$.

\begin{lemma}\label{lemma-estimateeta}
For every $\varepsilon>0$ there exists $\eta>0$ such that

\[|\int^\infty_{-\infty}{\langle\dot S_{\lambda_0}(t)u(t),u(t)\rangle\, dt}-\int^\eta_{-\eta}{\langle\dot S_{\lambda_0}(t)u(t),u(t)\rangle\, dt}|\leq \varepsilon \|u\|^2_{H},\quad u\in\ker\mathcal{A}_{\lambda_0}.\]
\end{lemma}

\begin{proof}
We estimate

\begin{align}\label{estimateeta}
\begin{split}
&|\int^\infty_{-\infty}{\langle\dot S_{\lambda_0}(t)u(t),u(t)\rangle\, dt}-\int^\eta_{-\eta}{\langle\dot S_{\lambda_0}(t)u(t),u(t)\rangle\, dt}|\\
&\leq|\int^\infty_{\eta}{\langle\dot S_{\lambda_0}(t)u(t),u(t)\rangle\, dt}|+|\int^{-\eta}_{-\infty}{\langle\dot S_{\lambda_0}(t)u(t),u(t)\rangle\, dt}|
\end{split}
\end{align}
and now consider the first term on the right hand side. From \eqref{growth}, we obtain for any $u\in H$

\begin{align}\label{estimate}
|\int^\infty_{\eta}{\langle\dot S_{\lambda_0}(t)u(t),u(t)\rangle\, dt}|\leq C_1\|u\|^2_{L^2(\eta,\infty)}.
\end{align}
Now let $\{u_1,\ldots,u_m\}$ be the basis of $\ker\mathcal{A}_{\lambda_0}$ that we defined in the second step of the proof, and let $\eta$ be large enough such that $C_1\sum^m_{i=1}{\|u_i\|^2_{L^2(\eta,\infty)}}<\frac{\varepsilon}{2}$. Since every $u\in\ker\mathcal{A}_{\lambda_0}$ can be written as
$u=\sum^m_{i=1}{\alpha_iu_i}$, where $\sum^m_{i=1}{\alpha^2_i}=\|u\|^2_H$, we obtain from \eqref{estimate}

\begin{align*}
|\int^\infty_{\eta}{\langle\dot S_{\lambda_0}(t)u(t),u(t)\rangle\, dt}|&\leq C_1\|u\|^2_{L^2(\eta,\infty)}\leq C_1\left(\sum^m_{i=1}{|\alpha_i|\|u_i\|_{L^2(\eta,\infty)}} \right)^2\\
&\leq  C_1\left(\sum^m_{i=1}{|\alpha_i|^2}\right)\left(\sum^m_{i=1}{\|u_i\|^2_{L^2(\eta,\infty)}} \right)<\frac{\varepsilon}{2}\|u\|^2_H.
\end{align*}  
Clearly, the assertion now follows by estimating the second term in \eqref{estimateeta} in a similar way and taking a larger $\eta>0$ if necessary.
\end{proof}

Hence there exists $\eta_0>0$ such that for all $\eta>\eta_0$ the quadratic form 

\[Q_1:\ker\mathcal{A}_{\lambda_0}\rightarrow\mathbb{R},\quad Q_1[u]=\int^\eta_{-\eta}{\langle\dot S_{\lambda_0}(t)u(t),u(t)\rangle\, dt}\]
has the property

\begin{align}\label{estimates}
&\|\Gamma(\mathcal{A},\lambda_0)-Q_1\|<\|L^{-1}_{\mathcal{A}}\|^{-1}<2\,\|L^{-1}_1\|^{-1},
\end{align}
where $L_\mathcal{A}$ and $L_1$ denote the symmetric linear operators on $\ker\mathcal{A}_{\lambda_0}$ that represent the quadratic forms $\Gamma(\mathcal{A},\lambda_0)$ and $L_1$ (cf. Appendix A).
Note that by Lemma \ref{quadraticnondeg}, the first inequality in \eqref{estimates} in particular implies that $Q_1$ is non-degenerate and $\sgn Q_1=\sgn\Gamma(\mathcal{A},\lambda_0)$.\\
Moreover, since the base functions $u_i\in\ker\mathcal{A}_{\lambda_0}$, $1\leq i\leq m$, that we defined in the second step of the proof, decay to $0$ exponentially for $t\rightarrow\pm\infty$ (cf. \eqref{decay}), we may assume that $\eta$ is sufficiently large such that

\begin{align}\label{estimatesII}
&\max_{i=1,\ldots,m}\|u_i(\pm\eta)\|^2<\frac{1}{16m^2C^8_2}\|L^{-1}_{\mathcal{A}}\|^{-1}.
\end{align}
Finally, by using the assertion of Lemma \ref{lemma-estimateeta} in the special case that $\dot{S}_{\lambda_0}$ is the identity, we obtain that $(1-\varepsilon)\|u\|^2_H\leq\|u\|^2_{L^2([-\eta,\eta],\mathbb{R}^{2n})}$ for all $u\in\ker\mathcal{A}_{\lambda_0}$. Accordingly, for $\eta$ sufficiently large, we have that

\begin{align}\label{estimatesIII}
\|u\|_H\leq2\,\|u\|_{L^2([-\eta,\eta],\mathbb{R}^{2n})},\quad u\in\ker\mathcal{A}_{\lambda_0}.
\end{align}
From now on we let $\eta>0$ be fixed such that \eqref{estimates}, \eqref{estimatesII} and \eqref{estimatesIII} hold.

%%%%%%%%%%%%%%%%%%%%%%%%%%%%%%%%%%%%%%%%%%%%%%%%%%%%%%%%%%%%%%%%%%%%%%%%%%%%%%%%%%%%%%%%%%%%%%%%%%%%%%%%%%%%%%%%%%%%%%%%%%%%%%%%%%%%%%%%%%%%%%%%%%%%%%%%%%%%%%%%%%%%%%%%%%%%%%%%%%%%%%%%%%%%%%%%%%%%%%%%%%%%%%%%%%%%%%%%%%%%%%%%%%%%%%%%%%%%%%%%%%%%%%%%%%%%%%%%%%%%%%%%%%%%%%%%%%%%%%%%%%%%%%%%%%%%%%%%%%%%%%%%%%%%%%%%%%%%%%%%%%%%%%%%%%%%%%%%%%%%%%%%%%%%%%%%%%%%%%%%%%

\subsubsection*{Step 4: The operators $A_\lambda$ and the quadratic forms $Q_2$}
Note that in the previous step we have reduced the integration to a finite interval, however, the functions are still defined on the whole real line. The aim of this third step is to reduce the setting to functions that are defined on a finite interval.\\
Before we perform this reduction, we introduce a family of operators that we will also need below in a subsequent step of our proof. We consider the restriction map $r:H^1(\mathbb{R},\mathbb{R}^{2n})\rightarrow H^1([-\eta,\eta],\mathbb{R}^{2n})$, and set 

\[A_\lambda:=r\mid_{\ker\mathcal{A}_{\lambda}}:\ker\mathcal{A}_{\lambda}\rightarrow L^2([-\eta,\eta],\mathbb{R}^{2n}),\] 
which is an injective linear map. Indeed, $A_\lambda u=0$ means that $u$ is a solution of $Ju'+S_{\lambda}u=0$ that vanishes on an open subset of the real line, which clearly implies $u=0\in\ker\mathcal{A}_\lambda$.\\
We now introduce finite dimensional subspaces of $L^2([-\eta,\eta],\mathbb{R}^{2n})$ by

\begin{align}\label{U}
\begin{split}
U_\lambda=\{u\in H^1([-\eta,\eta],\mathbb{R}^{2n}):&\, Ju'(t)+S_{\lambda}(t)u(t)=0,\,t\in[-\eta,\eta],\\
& u(-\eta)\in E^u_\lambda(-\eta),\, u(\eta)\in E^s_\lambda(\eta)\}\subset L^2([-\eta,\eta],\mathbb{R}^{2n}).
\end{split}
\end{align}

\begin{lemma}
The image of $A_\lambda$ is $U_\lambda$, $\lambda\in I$.
\end{lemma}

\begin{proof}
We first note that $\im A_\lambda\subset U_\lambda$ since $u(t)\rightarrow 0$, $t\rightarrow\pm\infty$, for any $u\in\ker\mathcal{A}_{\lambda}$. Moreover, it is clear that any element in $U_\lambda$ can be extended to a solution of \eqref{Hamiltonian} on $\mathbb{R}$. Since solutions of \eqref{Hamiltonian} belong to $H^1(\mathbb{R},\mathbb{R}^{2n})$ (cf. \eqref{decay}) and consequently to the kernel of $\mathcal{A}_\lambda$, we conclude that $A_\lambda$ is surjective onto $U_\lambda$.
\end{proof}

Let us note for later reference that $A_\lambda\neq 0$ if and only if $\lambda=\lambda_0$, and

\begin{align}\label{normA}
\|A_{\lambda_0}\|= 1.
\end{align}
Moreover, we conclude from \eqref{estimatesIII} that

\begin{align}\label{normAinv}
\|A^{-1}_{\lambda_0}\|\leq 2.
\end{align}
We now introduce a quadratic form by 

\[Q_2:U_{\lambda_0}\rightarrow\mathbb{R},\quad Q_2[u]=\langle\dot S_{\lambda_0}u,u\rangle_{L^2([-\eta,\eta],\mathbb{R}^n)}=\int^\eta_{-\eta}{\langle\dot S_{\lambda_0}(t)u(t),u(t)\rangle\, dt},\]
and note that $Q_1[u]=Q_2[A_{\lambda_0}u]$, $u\in\ker\mathcal{A}_{\lambda_0}$. Since $A_{\lambda_0}:\ker\mathcal{A}_{\lambda_0}\rightarrow U_{\lambda_0}$ is an isomorphism, we conclude that $Q_2$ is non-degenerate and $\sgn Q_2=\sgn Q_1$.

%%%%%%%%%%%%%%%%%%%%%%%%%%%%%%%%%%%%%%%%%%%%%%%%%%%%%%%%%%%%%%%%%%%%%%%%%%%%%%%%%%%%%%%%%%%%%%%%%%%%%%%%%%%%%%%%%%%%%%%%%%%%%%%%%%%%%%%%%%%%%%%%%%%%%%%%%%%%%%%%%%%%%%%%%%%%%%%%%%%%%%%%%%%%%%%%%%%%%%%%%%%%%%%%%%%%%%%%%%%%%%%%%%%%%%%%%%%%%%%%%%%%%%%%%%%%%%%%%%%%%%%%%%%%%%%%%%%%%%%%%%%%%%%%%%%%%%%%%%%%%%%%%%%%%%%%%%%%%%%%%%%%%%%%%%%%%%%%%%%%%%%%%%%%%%%%%%%%%%%%%%

\subsubsection*{Step 5: The operators $\widetilde{\mathcal{A}}$.}
We defined maps $\phi:I\times I\times[0,\infty)\rightarrow\Sp(2n)$ in the first step of our proof, and we now set

\[\psi:I\times [-\eta,\eta]\rightarrow\Sp(2n),\quad \psi(\lambda,t)=\phi\left(\frac{1}{2\eta}(t+\eta),\lambda,\eta\right).\]
Note that $\psi(\lambda,-\eta)\ell_1(\eta)=E^u_\lambda(-\eta)$ and $\psi(\lambda,\eta)\ell_2(\eta)=E^s_\lambda(\eta)$, and let us write for notational convenience $\psi_\lambda(t):=\psi(\lambda,t)$.\\
We now consider

\[\widetilde{W}:=\{u\in H^1([-\eta,\eta],\mathbb{R}^{2n}):\, u(-\eta)\in\ell_1(\eta),\, u(\eta)\in\ell_2(\eta)\},\quad \widetilde{H}=L^2([-\eta,\eta],\mathbb{R}^{2n}),\]
and the family of operators

\[\widetilde{\mathcal{A}}_\lambda:\widetilde{W}\rightarrow\widetilde{H},\quad\widetilde{\mathcal{A}}_\lambda u=Ju'(t)+\widetilde{S}_\lambda(t)u(t),\]
where
\[\widetilde S_\lambda(t)=\psi_\lambda(t)^\ast J\psi'_\lambda(t)+\psi_\lambda(t)^\ast S_\lambda(t)\psi_\lambda(t),\quad (\lambda,t)\in I\times [-\eta,\eta].\]
Note that $\psi_\lambda(t)^\ast J\psi_\lambda(t)=J$ implies

\begin{align}\label{vanishing}
\psi'_\lambda(t)^\ast J\psi_\lambda(t)+\psi_\lambda^\ast(t) J\psi'_\lambda(t)=0,\quad \lambda\in I,
\end{align}
and hence $\widetilde S_\lambda(t)^\ast=\widetilde S_\lambda(t)$. From $\ell_1(\eta),\ell_2(\eta)\in\Lambda(n)$, we see that $\widetilde{\mathcal{A}}$ is a path in $\mathcal{FS}(\widetilde{W},\widetilde{H})$. For later reference, let $\widetilde{\Psi}_\lambda:[-\eta,\eta]\rightarrow\Sp(2n)$ denote the solution of the initial value problem

\begin{equation}\label{tildepsi}
\left\{
\begin{aligned}
J\widetilde{\Psi}'_{\lambda}(t)+\widetilde{S}_\lambda(t)\widetilde{\Psi}_{\lambda}(t)&=0,\quad t\in[-\eta,\eta]\\
\widetilde{\Psi}_{\lambda}(-\eta)&=I_{2n}.
\end{aligned}
\right.
\end{equation}

\begin{lemma}\label{isoB}
For each $\lambda\in I$, the map 

\[L^2([-\eta,\eta],\mathbb{R}^{2n})\rightarrow L^2([-\eta,\eta],\mathbb{R}^{2n}),\quad u\mapsto \psi_{\lambda}u\]
defines an isomorphism $B_\lambda$ between $\ker\widetilde{\mathcal{A}}_{\lambda}$ and the space $U_\lambda$ introduced in \eqref{U}, such that

\begin{align}\label{normBinv}
\|B_\lambda\|\leq C_2\quad\text{and}\quad\|B^{-1}_\lambda\|\leq C_2,\quad\lambda\in I.
\end{align}
\end{lemma}

\begin{proof}
Let us assume that $u\in\ker\widetilde{\mathcal{A}}_{\lambda}$, i.e. 

\[Ju'(t)+\widetilde{S}_{\lambda}(t)u(t)=0,\,\lambda\in I,\quad u(-\eta)\in\ell_1(\eta),\,u(\eta)\in\ell_2(\eta).\] 
We obtain $(B_\lambda u)(-\eta)\in E^u_\lambda(-\eta)$, $(B_\lambda u)(\eta)\in E^s_\lambda(\eta)$, as well as 

\begin{align*}
&J(B_\lambda u)'(t)+S_{\lambda}(t)(B_\lambda u)(t)=J\psi'_{\lambda}(t)u(t)+J\psi_{\lambda}(t)u'(t)+S_{\lambda}(t)\psi_{\lambda}(t)u(t)\\
&=(\psi_{\lambda}(t)^\ast)^{-1}\left(Ju'(t)+\psi_{\lambda}(t)^\ast J\psi'_{\lambda}(t)u(t)+\psi_{\lambda}(t)^\ast S_{\lambda}(t)\psi_{\lambda}(t)u(t)\right)=0,\quad t\in [-\eta,\eta].
\end{align*}
Hence $B_\lambda u\in U_\lambda$. A similar computation shows that $B^{-1}_\lambda$ maps $U_\lambda$ into $\ker\widetilde{\mathcal{A}}_{\lambda}$, and finally, the estimates \eqref{normBinv} follow from \eqref{estimatespsi}.
\end{proof}

Since $U_{\lambda}$ is isomorphic to $\ker\mathcal{A}_\lambda$ by Step 3, we conclude that $\widetilde{\mathcal{A}}_\lambda$ is invertible if and only if $\lambda\neq\lambda_0$. Consequently, the spectral flow $\sfl(\widetilde{\mathcal{A}})$ is defined and $\lambda_0$ is the only crossing of $\widetilde{\mathcal{A}}$. The associated crossing form is

\[\Gamma(\widetilde{\mathcal{A}},\lambda_0)[u]=\int^\eta_{-\eta}{\left\langle\dot{\widetilde{S}}_{\lambda_0}(t)u(t),u(t)\right\rangle\, dt},\quad u\in\ker\widetilde{\mathcal{A}}_{\lambda_0}.\]
We now define a quadratic form on $\ker\widetilde{\mathcal{A}}_{\lambda_0}$ by

\[Q_3:\ker\widetilde{\mathcal{A}}_{\lambda_0}\rightarrow\mathbb{R},\quad Q_3[u]=Q_2[B_{\lambda_0}u],\]
and note that $Q_3$ is non-degenerate and has the same signature than $\Gamma(\mathcal{A},\lambda_0)$ by the previous steps of the proof. The following lemma shows that $\sfl(\mathcal{A})=\sfl(\widetilde{\mathcal{A}})$ (cf. Lemma \ref{quadraticnondeg}).

\begin{lemma}
There is a quadratic form $Q_4$ on $\ker\widetilde{\mathcal{A}}_{\lambda_0}$ such that 

\[\Gamma(\widetilde{\mathcal{A}},\lambda_0)[u]=Q_3[u]+Q_4[u],\quad u\in\ker\widetilde{\mathcal{A}}_{\lambda_0},\]
and $\|L_4\|<\|L^{-1}_3\|^{-1}$, where $L_3$ and $L_4$ denote the representations of $Q_3$ and $Q_4$, respectively.
\end{lemma}

\begin{proof}
From

\begin{align}\label{vanishingII}
\dot{\psi}_{\lambda_0}(t)^\ast J\psi_{\lambda_0}(t)+\psi_{\lambda_0}(t)^\ast J\dot{\psi}_{\lambda_0}(t)=0,\quad t\in[-\eta,\eta],
\end{align}
we obtain for $u\in\ker\widetilde{\mathcal{A}}_{\lambda_0}$ and $t\in[-\eta,\eta]$

\begin{align*}
\langle\dot{\widetilde{S}}_{\lambda_0}(t)u(t),u(t)\rangle&=\langle \dot \psi_{\lambda_0}(t)^\ast J\psi'_{\lambda_0}(t)u(t)+\psi_{\lambda_0}(t)^\ast J\dot \psi'_{\lambda_0}(t)u(t),u(t)\rangle +\langle \dot \psi_{\lambda_0}(t)^\ast S_{\lambda_0}(t)\psi_{\lambda_0}(t)u(t)\\
&+\psi_{\lambda_0}(t)^\ast\dot S_{\lambda_0}(t)\psi_{\lambda_0}(t)u(t)+\psi_{\lambda_0}(t)^\ast S_{\lambda_0}(t)\dot \psi_{\lambda_0}(t)u(t) ,u(t)\rangle\\
&=\langle \psi_{\lambda_0}(t)^\ast\dot S_{\lambda_0}(t)\psi_{\lambda_0}(t)u(t),u(t)\rangle\\
&+\langle\dot{\psi}_{\lambda_0}(t)^\ast J\psi'_{\lambda_0}(t) u(t)+\dot{\psi}_{\lambda_0}(t)^\ast J\psi_{\lambda_0}(t) u'(t)+\dot \psi_{\lambda_0}(t)^\ast S_{\lambda_0}(t)\psi_{\lambda_0}(t)u(t),u(t) \rangle\\
&+\langle \psi_{\lambda_0}(t)^\ast J\dot{\psi}'_{\lambda_0}(t) u(t)+\psi_{\lambda_0}(t)^\ast J\dot{\psi}_{\lambda_0}(t) u'(t)+ \psi_{\lambda_0}(t)^\ast S_{\lambda_0}(t)\dot \psi_{\lambda_0}(t)u(t),u(t) \rangle\\
&=\langle \dot S_{\lambda_0}(t)\psi_{\lambda_0}(t)u(t),\psi_{\lambda_0}(t)u(t)\rangle+\langle J(B_{\lambda_0}u)'(t)+ S_{\lambda_0}(t) (B_{\lambda_0}u)(t),\dot{\psi}_{\lambda_0}u(t) \rangle\\
&+\langle  J(\dot{\psi}_{\lambda_0}u)'(t)+ S_{\lambda_0}(t)(\dot\psi_{\lambda_0}(t)u(t)),\psi_{\lambda_0}(t)u(t) \rangle.
\end{align*}  
Since the second term vanishes by the previous lemma, we obtain 

\begin{align*}
\Gamma(\widetilde{\mathcal{A}},\lambda_0)[u]&=Q_3[u]+\int^{\eta}_{-\eta}{\langle  J(\dot{\psi}_{\lambda_0}u)'(t)+ S_{\lambda_0}(t)\dot\psi_{\lambda_0}(t)u(t),\psi_{\lambda_0}(t)u(t) \rangle\,dt}\\
&=Q_3[u]+\int^{\eta}_{-\eta}{\langle  \dot{\psi}_{\lambda_0}(t)u(t),J(\psi_{\lambda_0}u)'(t)+S_{\lambda_0}(t)\psi_{\lambda_0}(t)u(t) \rangle\,dt}\\
&-\langle\dot{\psi}_{\lambda_0}(\eta)u(\eta),J\psi_{\lambda_0}(\eta)u(\eta)\rangle+\langle\dot{\psi}_{\lambda_0}(-\eta)u(-\eta),J\psi_{\lambda_0}(-\eta)u(-\eta)\rangle\\
&=Q_3[u]-\langle\dot{\psi}_{\lambda_0}(\eta)u(\eta),J\psi_{\lambda_0}(\eta)u(\eta)\rangle+\langle\dot{\psi}_{\lambda_0}(-\eta)u(-\eta),J\psi_{\lambda_0}(-\eta)u(-\eta)\rangle.
\end{align*}
Hence it remains to show that the representing operator $L_4$ of

\[Q_4:\ker\widetilde{\mathcal{A}}_{\lambda_0}\rightarrow\mathbb{R},\quad Q_4[u]=-\langle\dot{\psi}_{\lambda_0}(\eta)u(\eta),J\psi_{\lambda_0}(\eta)u(\eta)\rangle+\langle\dot{\psi}_{\lambda_0}(-\eta)u(-\eta),J\psi_{\lambda_0}(-\eta)u(-\eta)\rangle\]
has the required bound.\\
We first estimate $\|L^{-1}_3\|^{-1}$ from below, where $L_3$ denotes the representing operator of the quadratic form $Q_3$, i.e. the uniquely determined selfadjoint operator on $\ker\widetilde{\mathcal{A}}_{\lambda_0}$ such that $Q_3(u)=\langle L_3u,u\rangle_{L^2([-\eta,\eta],\mathbb{R}^{2n})}$. From $Q_3[u]=Q_1[A^{-1}_{\lambda_0}B_{\lambda_0}u]$, $u\in\ker\widetilde{\mathcal{A}}_{\lambda_0}$, we see that $L_3=(A^{-1}_{\lambda_0}B_{\lambda_0})^\ast L_1 A^{-1}_{\lambda_0}B_{\lambda_0}$, where $L_1:\ker\mathcal{A}_{\lambda_0}\rightarrow\ker\mathcal{A}_{\lambda_0}$ denotes the representing operator of $Q_1$. In particular, 

\begin{align}\label{bound}
\|L^{-1}_3\|\leq\|L^{-1}_1\|\|B^{-1}_{\lambda_0}\|^2\|A_{\lambda_0}\|^2
\end{align}
and the inverse of the right hand side is a lower bound for $\|L^{-1}_3\|^{-1}$.\\
Let us now estimate the norm of $L_4$. We obtain

\begin{align}\label{estimateQ4}
\begin{split}
|Q_4[u]|&\leq |\langle\dot{\psi}_{\lambda_0}(\eta)u(\eta),J\psi_{\lambda_0}(\eta)u(\eta)\rangle|+|\langle\dot{\psi}_{\lambda_0}(-\eta)u(-\eta),J\psi_{\lambda_0}(-\eta)u(-\eta)\rangle|\\
&=|\langle \psi_{\lambda_0}(\eta)^\ast J\dot{\psi}_{\lambda_0}(\eta)u(\eta),u(\eta)\rangle|+|\langle\psi_{\lambda_0}(-\eta)^\ast J\dot{\psi}_{\lambda_0}(-\eta)u(-\eta),u(-\eta)\rangle|\\
&\leq\|\psi_{\lambda_0}(\eta)^\ast J\dot{\psi}_{\lambda_0}(\eta)\|\|u(\eta)\|^2+\|\psi_{\lambda_0}(-\eta)^\ast J\dot{\psi}_{\lambda_0}(-\eta)\|\|u(-\eta)\|^2\\
&\leq C^2_2(\|u(\eta)\|^2+\|u(-\eta)\|^2), 
\end{split}
\end{align}
where we have used \eqref{estimatespsi} in the last inequality. Let us now assume that $\|u\|_H=1$ and take $\tilde{u}\in\ker\mathcal{A}_{\lambda_0}$ such that $B^{-1}_{\lambda_0}A_{\lambda_0}\tilde{u}=u$.
We obtain from \eqref{normAinv} and \eqref{normBinv}

\begin{align*}
\|\tilde{u}\|_H\leq \|B_{\lambda_0}\|\|A^{-1}_{\lambda_0}\|\|u\|_{L^2([-\eta,\eta],\mathbb{R}^{2n})}\leq 2C_2,
\end{align*}
which implies that

\begin{align*}
\|u(\pm\eta)\|&\leq \|\psi_{\lambda_0}(\pm\eta)^{-1}\tilde{u}(\pm\eta)\|\leq C_2\|\tilde{u}(\pm\eta)\|\leq mC_2\max_{i=1,\ldots m}\|u_i(\pm\eta)\|\,\|\tilde{u}\|_H\\
&\leq 2mC^2_2 \max_{i=1,\ldots m}\|u_i(\pm\eta)\|,
\end{align*}
by \eqref{estimatespsi} and Lemma \ref{lem:estimatevalueu}. Finally, we use \eqref{estimates}, \eqref{estimatesII}, \eqref{normA}, \eqref{normBinv} and \eqref{bound} to obtain from \eqref{estimateQ4} and the previous inequality

\begin{align*}
\|L_4\|&\leq8m^2C^6_2\max\|u_i(\pm\eta)\|^2<\frac{1}{2C^2_2}\|L^{-1}_{\mathcal{A}}\|^{-1}\\
&\leq\frac{1}{C^2_2} \|L^{-1}_1\|^{-1}\leq\frac{1}{\|L^{-1}_1\|\|B^{-1}_{\lambda_0}\|^2\|A_{\lambda_0}\|^2}\leq \|L^{-1}_3\|^{-1}.
\end{align*} 
\end{proof}

Consequently, we now need to compute the spectral flow of $\widetilde{\mathcal{A}}$ which is the topic of the next two steps.

 %%%%%%%%%%%%%%%%%%%%%%%%%%%%%%%%%%%%%%%%%%%%%%%%%%%%%%%%%%%%%%%%%%%%%%%%%%%%%%%%%%%%%%%%%%%%%%%%%%%%%%%%%%%%%%%%%%%%%%%%%%%%%%%%%%%%%%%%%%%%%%%%%%%%%%%%%%%%%%%%%%%%%%%%%%%%%%%%%%%%%%%%%%%%%%%%%%%%%%%%%%%%%%%%%%%%%%%%%%%%%%%%%%%%%%%%%%%%%%%%%%%%%%%%%%%%%%%%%%%%%%%%%%%%%%%%%%%%%%%%%%%%%%%%%%%%%%%%%%%%%%%%%%%%%%%%%%%%%%%%%%%%%%%%%%%%%%%%%%%%%%%%%%%%%%%%%%%%%%%%%%

\subsubsection*{Step 6: The operators $\overline{\mathcal{A}}$.}

Since $\Sp(2n)$ is a connected Lie group, there is a smooth map $\varphi:[-\eta,\eta]\rightarrow\Sp(2n)$ such that $\varphi(-\eta)^{-1}\ell_1(\eta)=\varphi(\eta)^{-1}\ell_2(\eta)=\{0\}\times\mathbb{R}^n$. We define

\[\overline{W}:=\{u\in H^1([-\eta,\eta],\mathbb{R}^{2n}):\, u(-\eta),u(\eta)\in\{0\}\times\mathbb{R}^n\},\]
\[\overline{S}_\lambda(t):=\varphi(t)^\ast J\varphi'(t)+\varphi(t)^\ast\widetilde{S}_\lambda(t)\varphi(t)\]
and note that $\overline{S}$ is a family of symmetric matrices (cf. \eqref{vanishing}). We now consider the family of selfadjoint Fredholm operators

\[\overline{\mathcal{A}}_\lambda:\overline{W}\rightarrow L^2([-\eta,\eta],\mathbb{R}^{2n}),\quad u\mapsto Ju'+\overline{S}_\lambda(t)u.\]
As in Lemma \ref{isoB}, it is readily seen that for every $\lambda\in [-\eta,\eta]$, the map

\[C_\lambda:\ker\overline{\mathcal{A}}_\lambda\rightarrow L^2([-\eta,\eta],\mathbb{R}^{2n}),\quad u\mapsto \varphi_\lambda u\]
is an isomorphism onto $\ker\widetilde{\mathcal{A}}_\lambda$, and accordingly, the only crossing of the path $\overline{\mathcal{A}}$ is $\lambda_0$. Since $\varphi$ does not depend on the parameter $\lambda$, we obtain for the corresponding crossing form

\begin{align*}
\Gamma(\overline{\mathcal{A}},\lambda_0)[u]&=\int^\eta_{-\eta}{\left\langle \dot{\overline{S}}_{\lambda_0}(t)u(t),u(t)\right\rangle dt}=\int^\eta_{-\eta}{\left\langle\varphi(t)^\ast \dot{\widetilde{S}}_{\lambda_0}(t)\varphi(t)u(t),u(t)\right\rangle dt}\\
&=\int^\eta_{-\eta}{\left\langle\dot{\widetilde{S}}_{\lambda_0}(t)\varphi(t)u(t),\varphi(t)u(t)\right\rangle dt}=\Gamma(\widetilde{\mathcal{A}},\lambda_0)[C_{\lambda_0} u],\quad u\in\ker\overline{\mathcal{A}}_{\lambda_0},
\end{align*}
and consequently,

\[\sfl(\widetilde{\mathcal{A}})=\sfl(\overline{\mathcal{A}}).\]

%%%%%%%%%%%%%%%%%%%%%%%%%%%%%%%%%%%%%%%%%%%%%%%%%%%%%%%%%%%%%%%%%%%%%%%%%%%%%%%%%%%%%%%%%%%%%%%%%%%%%%%%%%%%%%%%%%%%%%%%%%%%%%%%%%%%%%%%%%%%%%%%%%%%%%%%%%%%%%%%%%%%%%%%%%%%%%%%%%%%%%%%%%%%%%%%%%%%%%%%%%%%%%%%%%%%%%%%%%%%%%%%%%%%%%%%%%%%%%%%%%%%%%%%%%%%%%%%%%%%%%%%%%%%%%%%%%%%%%%%%%%%%%%%%%%%%%%%%%%%%%%%%%%%%%%%%%%%%%%%%%%%%%%%%%%%%%%%%%%%%%%%%%%%%%%%%%%%%%%%%%

\subsubsection*{Step 7: The spectral flow of $\overline{\mathcal{A}}$}

We let $\overline{\Psi}_\lambda:[-\eta,\eta]\rightarrow\Sp(2n)$ be the solution of the initial value problem

\begin{equation}\label{overlinepsi}
\left\{
\begin{aligned}
J\overline{\Psi}'_{\lambda}(t)+\overline{S}_\lambda(t)\overline{\Psi}_{\lambda}(t)&=0,\quad t\in[-\eta,\eta]\\
\overline{\Psi}_{\lambda}(-\eta)&=I_{2n},
\end{aligned}
\right.
\end{equation}
and we write

\begin{align*}
\overline{\Psi}_\lambda(t)=\begin{pmatrix}
a_\lambda(t)&b_\lambda(t)\\
c_\lambda(t)&d_\lambda(t)	
\end{pmatrix}\in\Sp(2n),\quad t\in[-\eta,\eta].
\end{align*}
The straightforward proof of the following lemma is left to the reader.

\begin{lemma}\label{kernelMaslov}
$u\in\ker\overline{\mathcal{A}}_\lambda$ if and only if there exists $v\in\ker b_\lambda(\eta)$ such that $u(t)=\overline{\Psi}_\lambda(t)(0,v)$, $t\in[-\eta,\eta]$.
\end{lemma}

In particular, we see from

\[\overline{\Psi}_\lambda(\eta)(\{0\}\times\mathbb{R}^n)\cap(\{0\}\times\mathbb{R}^n)=\{(0,d_\lambda(\eta)v):\, v\in\ker b_\lambda(\eta)\},\]
that $\overline{\Psi}_\lambda(\eta)(\{0\}\times\mathbb{R}^n)\cap(\{0\}\times\mathbb{R}^n)\neq\{0\}$ if and only if $\lambda=\lambda_0$. Hence, if $\lambda_0$ turns out to be a regular crossing of the curve $\overline{\Psi}_\cdot(\eta)(\{0\}\times\mathbb{R}^n)$ with respect to $\{0\}\times\mathbb{R}^n$, then the corresponding Maslov index is given by the signature of the crossing form $\Gamma(\Psi_\cdot(\eta)(\{0\}\times\mathbb{R}^n),\{0\}\times\mathbb{R}^n,\lambda_0)$.\\
The subsequent computation follows Lemma 7.2 in \cite{Robbin-Salamon}: From the identity

\[J\overline{\Psi}'_\lambda(t)+\overline{S}_\lambda(t)\overline{\Psi}_\lambda(t)=0,\quad t\in[-\eta,\eta],\quad \lambda\in I,\]
we infer the two equalities

\begin{align}\label{equI}
\overline{\Psi}'(t)^\ast J=\overline{\Psi}(t)^\ast\overline{S}_\lambda(t)
\end{align}
and

\begin{align}\label{equII}
\dot{\overline{S}}_\lambda(t)\overline{\Psi}_\lambda(t)+\overline{S}_\lambda(t)\dot{\overline{\Psi}}_\lambda(t)=-J\dot{\overline{\Psi}}'_\lambda(t).
\end{align}
If we multiply \eqref{equII} on the left by $\overline{\Psi}_\lambda(t)^\ast$ and integrate over $[-\eta,\eta]$, we obtain

\begin{align*}
\int^\eta_{-\eta}{\overline{\Psi}_\lambda(t)^\ast\dot{\overline{S}}_\lambda(t)\overline{\Psi}_\lambda(t)\, dt}&=-\int^\eta_{-\eta}{\overline{\Psi}_\lambda(t)^\ast J\dot{\overline{\Psi}}'_\lambda(t)\, dt}-\int^\eta_{-\eta}{\overline{\Psi}_\lambda(t)^\ast \overline{S}_\lambda(t)\dot{\overline{\Psi}}_\lambda(t)\, dt}\\
&=-\int^\eta_{-\eta}{\overline{\Psi}_\lambda(t)^\ast J\dot{\overline{\Psi}}'_\lambda(t)\, dt}-\int^\eta_{-\eta}{\overline{\Psi}'_\lambda(t)^\ast J\dot{\overline{\Psi}}_\lambda(t)\, dt}\\
&=-\overline{\Psi}_\lambda(\eta)^\ast J\dot{\overline{\Psi}}_\lambda(\eta)-\overline{\Psi}_\lambda(-\eta)^\ast J\dot{\overline{\Psi}}_\lambda(-\eta)\\
&=-\overline{\Psi}_\lambda(\eta)^\ast J\dot{\overline{\Psi}}_\lambda(\eta),
\end{align*}
where we have used in the last equality that $\dot{\overline{\Psi}}_\lambda(-\eta)=0$ since $\overline{\Psi}_\lambda(-\eta)$ is constant.\\
Let us now take $u\in\ker\overline{\mathcal{A}}_{\lambda_0}$ and $v\in\ker b_{\lambda_0}(\eta)$ such that $u(t)=\overline{\Psi}_{\lambda_0}(t)(0,v)$, $t\in[-\eta,\eta]$, as in Lemma \ref{kernelMaslov}. We obtain

\begin{align*}
\Gamma(\overline{\mathcal{A}},\lambda_0)[u]&=\int^\eta_{-\eta}{\left\langle \dot{\overline{S}}_{\lambda_0}(t)\overline{\Psi}_{\lambda_0}(t)(0,v),\overline{\Psi}_{\lambda_0}(t)(0,v)\right\rangle dt}\\
&=\int^\eta_{-\eta}{\left\langle\overline{\Psi}_{\lambda_0}(t)^\ast \dot{\overline{S}}_{\lambda_0}(t)\overline{\Psi}_{\lambda_0}(t)(0,v),(0,v)\right\rangle dt}\\
&=-\left\langle\overline{\Psi}_{\lambda_0}(\eta)^\ast J\dot{\overline{\Psi}}_{\lambda_0}(\eta)(0,v),(0,v)\right\rangle=-\left\langle J\dot{\overline{\Psi}}_{\lambda_0}(\eta)(0,v),\overline{\Psi}_{\lambda_0}(\eta)(0,v)\right\rangle\\
&=\langle\dot{d}_{\lambda_0}(\eta)v,b_{\lambda_0}(\eta)v\rangle-\langle d_{\lambda_0}(\eta)v,\dot b_{\lambda_0}(\eta)v\rangle=-\langle d_{\lambda_0}(\eta)v,\dot b_{\lambda_0}(\eta)v\rangle.
\end{align*}
It follows from Lemma \ref{kernelMaslov} that the right hand side is a non-degenerate quadratic form on $\ker b_{\lambda_0}$, and by \eqref{MaslovSympMat}, its signature is the Maslov index of the curve $\overline{\Psi}_\cdot(\{0\}\times\mathbb{R}^n)$ relative to $\{0\}\times\mathbb{R}^n$. Thus we have shown that

\[\sfl(\overline{\mathcal{A}})=\mu_{Mas}(\overline{\Psi}_\cdot(\eta)(\{0\}\times\mathbb{R}^n),\{0\}\times\mathbb{R}^n).\]

%%%%%%%%%%%%%%%%%%%%%%%%%%%%%%%%%%%%%%%%%%%%%%%%%%%%%%%%%%%%%%%%%%%%%%%%%%%%%%%%%%%%%%%%%%%%%%%%%%%%%%%%%%%%%%%%%%%%%%%%%%%%%%%%%%%%%%%%%%%%%%%%%%%%%%%%%%%%%%%%%%%%%%%%%%%%%%%%%%%%%%%%%%%%%%%%%%%%%%%%%%%%%%%%%%%%%%%%%%%%%%%%%%%%%%%%%%%%%%%%%%%%%%%%%%%%%%%%%%%%%%%%%%%%%%%%%%%%%%%%%%%%%%%%%%%%%%%%%%%%%%%%%%%%%%%%%%%%%%%%%%%%%%%%%%%%%%%%%%%%%%%%%%%%%%%%%%%%%%%%%%

\subsubsection*{Step 8: The final argument}
We first note that it is easily seen from 

\[(\varphi^{-1}(t))'=-\varphi(t)^{-1}\varphi'(t)\varphi(t)^{-1}\quad \text{and}\quad \varphi(t)^{-1}=-J\varphi(t)^\ast J,\quad t\in[-\eta,\eta],\]
that the fundamental matrices $\widetilde{\Psi}$ introduced in \eqref{tildepsi} and $\overline{\Psi}$ introduced in \eqref{overlinepsi} are related by

\[\overline{\Psi}_\lambda(t)=\varphi(t)^{-1}\widetilde{\Psi}_\lambda(t)\varphi(-\eta),\quad t\in[-\eta,\eta],\,\,\lambda\in I.\]
We obtain from \eqref{Maslovnatural}

\begin{align*}
\mu_{mas}(\overline{\Psi}_\cdot(\eta)(\{0\}\times\mathbb{R}^n),\{0\}\times\mathbb{R}^n)&=\mu_{mas}(\varphi(\eta)^{-1}\widetilde{\Psi}_\cdot(\eta){\varphi}(-\eta)(\{0\}\times\mathbb{R}^n),\{0\}\times\mathbb{R}^n)\\
&=\mu_{mas}(\widetilde{\Psi}_\cdot(\eta)\ell_1(\eta),\ell_2(\eta)).
\end{align*}
Analogously, we have for $\widetilde{\Psi}$ and $\Psi$

\[\widetilde{\Psi}_\lambda(t)=\psi^{-1}_\lambda(t)\Psi_{(\lambda,-\eta)}(t)\psi_\lambda(-\eta),\quad \lambda\in I,\,\,t\in[-\eta,\eta],\]	
and we conclude by using \eqref{Maslovnatural} again that

\begin{align*}
\mu_{mas}(\widetilde\Psi_\cdot(\eta)\ell_1(\eta),\ell_2(\eta))&=\mu_{mas}(\Psi_{(\cdot,-\eta)}(\eta) \psi_\cdot(-\eta)\ell_1(\eta),\psi_\cdot(\eta)\ell_2(\eta))\\
&=\mu_{mas}(\Psi_{(\cdot,-\eta)}(\eta) E^u_\cdot(-\eta),E^s_\cdot(\eta)).
\end{align*}
Finally, Theorem \ref{main} follows from

\[\Psi_{(\lambda,-\eta)}(\eta) E^u_\lambda(-\eta)=E^u_\lambda(\eta),\quad\lambda\in I,\]
and

\[\mu_{mas}(E^u_\cdot(\eta),E^s_\cdot(\eta))=\mu_{mas}(E^u_\cdot(0),E^s_\cdot(0)),\]
which was shown in Lemma \ref{Maslovconstant}.

%%%%%%%%%%%%%%%%%%%%%%%%%%%%%%%%%%%%%%%%%%%%%%%%%%%%%%%%%%%%%%%%%%%%%%%%%%%%%%%%%%%%%%%%%%%%%%%%%%%%%%%%%%%%%%%%%%%%%%%%%%%%%%%%%%%%%%%%%%%%%%%%%%%%%%%%%%%%%%%%%%%%%%%%%%%%%%%%%%%%%%%%%%%%%%%%%%%%%%%%%%%%%%%%%%%%%%%%%%%%%%%%%%%%%%%%%%%%%%%%%%%%%%%%%%%%%%%%%%%%%%%%%%%%%%%%%%%%%%%%%%%%%%%%%%%%%%%%%%%%%%%%%%%%%%%%%%%%%%%%%%%%%%%%%%%%%%%%%%%%%%%%%%%%%%%%%%%%%%%%%%

\section{Application to bifurcation of homoclinic solutions}\label{sect:bifurcation}

Let $\mathcal{H}:I\times\mathbb{R}\times\mathbb{R}^{2n}\rightarrow\mathbb{R}$ be a continuous map such that $\mathcal{H}_\lambda:=\mathcal{H}(\lambda,\cdot,\cdot):\mathbb{R}\times\mathbb{R}^{2n}\rightarrow\mathbb{R}$ is $C^2$ for all $\lambda\in I$ and its derivatives depend continuously on $\lambda\in I$. We consider the family of Hamiltonian systems

\begin{equation}\label{Hamiltoniannonlin}
\left\{
\begin{aligned}
Ju'(t)+\nabla_u \mathcal{H}_\lambda(t,u(t))&=0,\quad t\in\mathbb{R}\\
\lim_{t\rightarrow\pm\infty}u(t)&=0,
\end{aligned}
\right.
\end{equation}
where $\nabla_u$ denotes the gradient with respect to the variable $u\in\mathbb{R}^{2n}$, and $J$ is the standard symplectic matrix \eqref{J}. In what follows, we assume that

\begin{align}\label{Hgrowth}
\mathcal{H}_\lambda(t,u)=\frac{1}{2}\langle S_\lambda(t)u,u\rangle+G(\lambda,t,u),
\end{align}
where $S:I\times\mathbb{R}\rightarrow\mathcal{S}(\mathbb{R}^{2n})$ is a smooth family of symmetric matrices as in Section \ref{sect:main}, $G(\lambda,t,u)$ vanishes up to second order at $u=0$, and there are $p>0$, $C\geq 0$ and $g\in H^1(\mathbb{R},\mathbb{R})$ such that

\[|D^2_uG(\lambda,t,u)|\leq g(t)+C|u|^p.\]

Note that in particular $\nabla_u \mathcal{H}_\lambda(t,0)=0$ for all $(\lambda,t)\in I\times\mathbb{R}$, so that $u\equiv 0$ is a solution of \eqref{Hamiltoniannonlin} for all $\lambda\in I$. By definition, \textit{homoclinics} are solutions $(\lambda,u)$ of \eqref{Hamiltoniannonlin} such that $u\not\equiv 0$.\\
The linearisation of \eqref{Hamiltoniannonlin} at the trivial solution $u\equiv 0$ is

\begin{equation}\label{Hamiltonianlin}
\left\{
\begin{aligned}
Ju'(t)+S_\lambda(t)u(t)&=0,\quad t\in\mathbb{R}\\
\lim_{t\rightarrow\pm\infty}u(t)&=0,
\end{aligned}
\right.
\end{equation}
where $S_\lambda(t)$ is the Hessian of $\mathcal{H}_\lambda(t,\cdot)$ at the critical point $0\in\mathbb{R}^{2n}$ according to \eqref{Hgrowth}. In what follows, we assume without further reference that the family $S_\lambda$ has limits for $t\rightarrow\pm\infty$ as in Section \ref{sect:main}.\\
Let $C^1_0(\mathbb{R},\mathbb{R}^{2n})$ be the Banach space of all continuously differentiable $\mathbb{R}^{2n}$-valued functions $u$ such that $u$ and $u'$ vanish at infinity, where the norm is defined by

\[\|u\|=\sup_{t\in\mathbb{R}}|u(t)|+\sup_{t\in\mathbb{R}}|u'(t)|.\]

\begin{defi}\label{defibif}
We call $\lambda^\ast\in I$ a bifurcation point for homoclinic solutions from the stationary branch if every neighbourhood of $(\lambda^\ast,0)\in I\times C^1_0(\mathbb{R},\mathbb{R}^{2n})$ contains a non-trivial solution $(\lambda,u)$ of \eqref{Hamiltoniannonlin}.
\end{defi}

We denote by $\Sigma$ the set of all $\lambda\in I$ such that the linearised equation \eqref{Hamiltonianlin} has a non-trivial solution. We will explain below the obvious fact that the set of all bifurcation points is contained in $\Sigma$.
The main theorem of this section reads as follows:

\begin{theorem}\label{bif}
Assume that $\mathcal{H}$ is of the form \eqref{Hgrowth} and that A1) and A2) hold for the linearised equation \eqref{Hamiltonianlin}. Let $E^s_\lambda(0)$ and $E^u_\lambda(0)$, $\lambda\in I$, denote the stable and unstable subspaces of \eqref{Hamiltonianlin}. 

\begin{itemize}
	\item[i)] If $\mu_{mas}(E^u_\cdot(0),E^s_\cdot(0))\neq 0$, then there exists at least one bifurcation point $\lambda^\ast\in(0,1)$ for homoclinic solutions of \eqref{Hamiltoniannonlin}.
	\item[ii)] If $\Sigma$ is finite, then there are at least
	
	\[\left\lfloor\frac{|\mu_{mas}(E^u_\cdot(0),E^s_\cdot(0))|}{n}\right\rfloor\]
	distinct bifurcation points in $(0,1)$, where $\lfloor\cdot\rfloor$ denotes the integral part of a real number.   
\end{itemize}
\end{theorem}
Let us point out that $\Sigma$ is finite, for example, if the corresponding path $\mathcal{A}$ in \eqref{diffop} has only regular crossings.\\
The proof of Theorem \ref{bif} weakens the assumption on the differentiability of the Hamiltonian $\mathcal{H}$ from Theorem 2.2 in \cite{Jacobo} by using the recent work \cite{JacBifIch} of Pejsachowicz and the author. Concretely, here we merely suppose that $\mathcal{H}$ depends continuously on the parameter $\lambda$, whereas this was required to be smooth in \cite{Jacobo}.\\
Let us first recall the main theorem of \cite{JacBifIch}. Let $H$ be a separable Hilbert space and $f:I\times H\rightarrow\mathbb{R}$ a continuous function such that each $f_\lambda:=f(\lambda,\cdot):H\rightarrow\mathbb{R}$ is $C^2$ and its first and second derivatives depend continuously on $\lambda\in I$. In what follows, we assume that $0\in H$ is a critical point of all $f_\lambda$, $\lambda\in I$. We call $\lambda\in I$ a bifurcation point of critical points of $f$ if every neighbourhood of $I\times\{0\}$ in $I\times H$ contains elements $(\lambda,u)$ such that $u\neq 0$ is a critical point of $f_\lambda$. The second derivatives $D^2_0f_\lambda$ at the critical point $0\in H$ are bounded symmetric bilinear forms and there exists a unique continuous path of selfadjoint operators $L$ on $H$ such that 

\begin{align}\label{Riesz}
D^2_0f_\lambda(u,v)=\langle L_\lambda u,v\rangle_H,\quad u,v\in H,\,\lambda\in I.
\end{align}
From the implicit function theorem, it is easily seen that $L_{\lambda^\ast}$ is not invertible if $\lambda^\ast$ is a bifurcation point of critical points of $f$.\\
Let us now assume that $L_\lambda$ is Fredholm for all $\lambda\in I$, so that the spectral flow of the path $L:I\rightarrow\mathcal{FS}(H)$ of bounded selfadjoint Fredholm operators on $H$ is defined (cf. Section \ref{sect:sfl}). The main theorem of \cite{JacBifIch} reads as follows:

\begin{theorem}\label{mainbif}   
If $\sfl(L)\neq 0$ and $L_0,L_1$ are invertible, then there exists a bifurcation point $\lambda^\ast\in (0,1)$ of critical points of $f$ from the trivial branch.
\end{theorem}

Moreover, if there is an a priori bound on the dimension of the kernels of the operators $L_\lambda$, then the number of bifurcation points can be estimated from below as follows (cf. \cite[Thm. 2.1 ii)]{JacBifIch}).

\begin{theorem}\label{biffinite}
Assume that there exist only finitely many $\lambda\in I$ such that $\ker L_\lambda\neq 0$. Then there are at least

\[\left\lfloor\frac{|\sfl(L)|}{\max_{\lambda\in I}\dim\ker L_\lambda}\right\rfloor\]
distinct bifurcations of critical points from the trivial branch $I\times\{0\}$.
\end{theorem}

Let us now briefly recall the variational formulation of the equations \eqref{Hamiltoniannonlin} from \cite[\S 4]{Jacobo}. The bilinear forms $b_\lambda(u,v)=\langle \mathcal{A}_\lambda u,v\rangle_{L^2}$, $u,v\in H^1(\mathbb{R},\mathbb{R}^{2n})$, extend to bounded forms on the well known fractional Sobolev space $H^\frac{1}{2}(\mathbb{R},\mathbb{R}^{2n})$, which can be described in terms of Fourier transforms (cf. eg. \cite[\S 10]{Stuart}). Under the assumption \eqref{Hgrowth}, 

\[f_\lambda:H^\frac{1}{2}(\mathbb{R},\mathbb{R}^{2n})\rightarrow\mathbb{R},\quad f_\lambda(u)=b_\lambda(u,u)+\int^\infty_{-\infty}{G(\lambda,t,u(t))\,dt}\]
are $C^2$-functionals such that $f:I\times H^\frac{1}{2}(\mathbb{R},\mathbb{R}^{2n})\rightarrow\mathbb{R}$ is continuous and all its derivatives depend continuously on $\lambda\in I$. A careful examination of $f$ shows that the critical points of $f_\lambda$ belong to $C^1_0(\mathbb{R},\mathbb{R}^{2n})$ and are the classical solutions of the differential equation \eqref{Hamiltoniannonlin}. Moreover, every bifurcation point of critical points of $f$ is a bifurcation point of \eqref{Hamiltoniannonlin} in the sense of Definition \ref{defibif}. Finally,
the second derivative of $f_\lambda$ at the critical point $0\in H^\frac{1}{2}(\mathbb{R},\mathbb{R}^{2n})$ is given by $D^2_0f_\lambda(u,v)=b_\lambda(u,v)$ and the corresponding operators $L_\lambda:H^\frac{1}{2}(\mathbb{R},\mathbb{R}^{2n})\rightarrow H^\frac{1}{2}(\mathbb{R},\mathbb{R}^{2n})$ in \eqref{Riesz} are Fredholm.\\
Let us now explain how Theorem \ref{bif} follows from Theorem \ref{mainbif} and Theorem \ref{biffinite}. We first note that by the definition of the form $b_\lambda$, it clearly follows that $\ker L_\lambda=\ker\mathcal{A}_\lambda$ for all $\lambda\in I$, and moreover, the latter space is contained in $C^1_0(\mathbb{R},\mathbb{R}^{2n})$. Since $L_\lambda$ is Fredholm of index $0$, we see that $L_\lambda$ is invertible if and only if \eqref{Hamiltonianlin} has no non-trivial solution; i.e. $\lambda\notin\Sigma$. This in particular shows that $\Sigma$ contains the set of all bifurcation points of \eqref{Hamiltoniannonlin}. Moreover, it follows from assumption A2) that the operators $L_0$ and $L_1$ are invertible. Our next aim is to prove that $\sfl(L)=\sfl(\mathcal{A})$, which implies the first assertion of Theorem \ref{bif} by our main Theorem \ref{main}. By definition

\begin{align}\label{equalcrossing}
\langle L_\lambda u,v\rangle_{H^\frac{1}{2}(\mathbb{R},\mathbb{R}^{2n})}=\langle\mathcal{A}_\lambda u,v\rangle_{L^2(\mathbb{R},\mathbb{R}^{2n})},\quad u\in H^1(\mathbb{R},\mathbb{R}^{2n}),\, v\in H^\frac{1}{2}(\mathbb{R},\mathbb{R}^{2n}),
\end{align}
which is enough to conclude the equality of $\sfl(L)$ and $\sfl(\mathcal{A})$ by the more general Theorem 2.6 of \cite{Hamiltonian} for the index bundle of families of selfadjoint Fredholm operators (cf. also \cite{MorseIch}). Here, however, we want to use the approach from Section \ref{sect:sfl} for the computation of the spectral flow by crossing forms.\\
We denote by $B:H^\frac{1}{2}(\mathbb{R},\mathbb{R}^{2n})\rightarrow H^\frac{1}{2}(\mathbb{R},\mathbb{R}^{2n})$ the unique selfadjoint operator such that\linebreak $\langle u,v\rangle_{L^2(\mathbb{R},\mathbb{R}^{2n})}=\langle Bu,v\rangle_{H^\frac{1}{2}(\mathbb{R},\mathbb{R}^{2n})}$ for all $u,v\in H^\frac{1}{2}(\mathbb{R},\mathbb{R}^{2n})$. For $\delta'>0$ sufficiently small, the selfadjoint operators $L^\delta_\lambda:=L_\lambda+\delta B$ are Fredholm and have invertible ends for all $0\leq\delta<\delta'$, so that $\sfl(L^\delta)$ is defined and is equal to $\sfl(L)$. By Theorem \ref{thm:Sard}, we can take $0\leq\delta<\delta'$ such that $\mathcal{A}^\delta=\mathcal{A}+\delta I_H$ has only regular crossings and $\sfl(\mathcal{A}^\delta)=\sfl(\mathcal{A})$. We conclude from \eqref{equalcrossing} and the definition of $B$ that 

\[\langle L^\delta_\lambda u,v\rangle_{H^\frac{1}{2}(\mathbb{R},\mathbb{R}^{2n})}=\langle\mathcal{A}^\delta_\lambda u,v\rangle_{L^2(\mathbb{R},\mathbb{R}^{2n})},\quad u\in H^1(\mathbb{R},\mathbb{R}^{2n}),\, v\in H^\frac{1}{2}(\mathbb{R},\mathbb{R}^{2n}),\]
which shows first that $\ker L^\delta_\lambda=\ker\mathcal{A}^\delta_\lambda$, and second that $\Gamma(L^\delta,\lambda)=\Gamma(\mathcal{A}^\delta,\lambda)$, $\lambda\in I$. Since $\mathcal{A}^\delta$ has only regular crossings, we see that the crossings of $L^\delta$ are regular as well. Consequently, we obtain from Theorem \ref{thm:crossing}

\[\sfl(\mathcal{A})=\sfl(\mathcal{A}^\delta)=\sum_{\lambda\in(0,1)}{\sgn\Gamma(\mathcal{A}^\delta,\lambda)}=\sum_{\lambda\in(0,1)}{\sgn\Gamma(L^\delta,\lambda)}=\sfl(L^\delta)=\sfl(L),\]
and the first part of Theorem \ref{bif} is shown.\\
The remaining assertion of Theorem \ref{bif} now follows from Theorem \ref{biffinite} if we can show that\linebreak $\dim\ker L_\lambda\leq n$ for all $\lambda\in I$. Since $\ker L_\lambda=\ker\mathcal{A}_\lambda$ and $\ker\mathcal{A}_\lambda$ is isomorphic to $E^u_\lambda(0)\cap E^s_\lambda(0)$ by \eqref{evaliso}, we obtain the required estimate from the fact that $\dim E^u_\lambda(0)=\dim E^s_\lambda(0)=n$ (cf. Lemma \ref{stabsympl}).

%%%%%%%%%%%%%%%%%%%%%%%%%%%%%%%%%%%%%%%%%%%%%%%%%%%%%%%%%%%%%%%%%%%%%%%%%%%%%%%%%%%%%%%%%%%%%%%%%%%%%%%%%%%%%%%%%%%%%%%%%%%%%%%%%%%%%%%%%%%%%%%%%%%%%%%%%%%%%%%%%%%%%%%%%%%%%%%%%%%%%%%%%%%%%%%%%%%%%%%%%%%%%%%%%%%%%%%%%%%%%%%%%%%%%%%%%%%%%%%%%%%%%%%%%%%%%%%%%%%%%%%%%%%%%%%%%%%%%%%%%%%%%%%%%%%%%%%%%%%%%%%%%%%%%%%%%%%%%%%%%%%%%%%%%%%%%%%%%%%%%%%%%%%%%%%%%%%%%%%%%%

\section*{Appendix}

\appendix

\section{Quadratic forms}
Let $H$ be a finite dimensional Hilbert space with norm $\|\cdot\|$. A quadratic form $Q:H\rightarrow\mathbb{R}$ on $H$ is a map for which there exists a bilinear form $b:H\times H\rightarrow\mathbb{R}$ such that

\[Q(u)=b(u,u),\quad u\in H.\]
The set $\mathcal{Q}(H)$ of all quadratic forms on $H$ is a Banach space with respect to the norm

\[\|Q\|=\sup_{\|u\|=1}|Q(u)|.\] 
For every $Q\in\mathcal{Q}(H)$ there is a unique selfadjoint linear operator $L_Q:H\rightarrow H$ such that 

\begin{align}\label{Rieszfinite}
Q(u)=\langle L_Qu,u\rangle,\quad u\in H.
\end{align}
Note that $Q$ is non-degenerate if and only if $L_Q$ is invertible, and moreover, $\|Q\|=\|L_Q\|$. If $H_1$ is another finite dimensional Hilbert space and $Q_1\in\mathcal{Q}(H_1)$, we write $Q_1\sim Q$ if there exists a an isomorphism $M:H_1\rightarrow H$ such that $L_{Q_1}=M^\ast L_QM$. Clearly, $Q_1$ is non-degenerate if and only if $Q_2$ is. The signature of a quadratic form $Q$ is defined by

\begin{align*}
\sgn Q=m^+(Q)-m^-(Q),
\end{align*}
where
\begin{align*}
m^\pm(Q)=\dim\bigoplus_{\pm\lambda>0}\{u\in H:\,L_Qu=\lambda u\}.
\end{align*}
Note that $\sgn Q_1=\sgn Q$ if $Q\sim Q_1$. It follows from the continuity of eigenvalues (cf. \cite[II.5.1]{Kato}) that the signature is constant on each connected component of the subspace of all non-degenerate quadratic forms. The following stability result is a simple consequence of the Neumann series (cf. \cite[I.4.4]{Kato}).

\begin{lemma}\label{quadraticnondeg}
If $Q_1$ is a non-degenerate quadratic form and $Q_2\in\mathcal{Q}(H)$ is such that $\|L_{Q_2}\|<\|L^{-1}_{Q_1}\|^{-1}$, then $Q_1+Q_2$ is non-degenerate and $\sgn(Q_1+Q_2)=\sgn Q_1$.
\end{lemma}

%%%%%%%%%%%%%%%%%%%%%%%%%%%%%%%%%%%%%%%%%%%%%%%%%%%%%%%%%%%%%%%%%%%%%%%%%%%%%%%%%%%%%%%%%%%%%%%%%%%%%%%%%%%%%%%%%%%%%%%%%%%%%%%%%%%%%%%%%%%%%%%%%%%%%%%%%%%%%%%%%%%%%%%%%%%%%%%%%%%%%%%%%%%%%%%%%%%%%%%%%%%%%%%%%%%%%%%%%%%%%%%%%%%%%%%%%%%%%%%%%%%%%%%%%%%%%%%%%%%%%%%%%%%%%%%%%%%%%%%%%%%%%%%%%%%%%%%%%%%%%%%%%%%%%%%%%%%%%%%%%%%%%%%%%%%%%%%%%%%%%%%%%%%%%%%%%%%%%%%%%%

\section{On liftings in $\Lambda(n)$}
We assume throughout that $X$, $Y$ and $B$ are paracompact Hausdorff spaces. Let us recall the following lifting property for fibre bundles (cf. \cite[Thm. 6.4]{Bredon}):

\begin{lemma}\label{lifting}
Let $p:Y\rightarrow B$ be a fibre bundle and $X$ a contractible space. Then for any continuous map $f:X\rightarrow B$ there exists a continuous map $F:X\rightarrow Y$ such that $p\circ F=f$.
\end{lemma}

We now consider the fibre bundle $p:\Sp(2n)\rightarrow \Lambda(n)$, where $p$ assigns to $A\in\Sp(2n)$ the Lagrangian subspace $A\ell_0\in\Lambda(n)$ for some fixed $\ell_0\in\Lambda(n)$ (cf. \cite[Rem. 1.4]{Robbin-SalamonMAS}).\\ 
Before we prove the main result of this section, we show the following simple application of Lemma \ref{lifting}.

\begin{lemma}
Let $X$ and $Y$ be contractible spaces. If $\gamma:X\times Y\rightarrow\Lambda(n)$ and $\ell_2:Y\rightarrow\Lambda(n)$ are continuous maps, then there exists a continuous map $\phi:X\times Y\rightarrow\Sp(2n)$ such that 

\[\phi(\lambda,\eta)\ell_2(\eta)=\gamma(\lambda,\eta),\quad(\lambda,\eta)\in X\times Y.\]
\end{lemma}

\begin{proof}
By Lemma \ref{lifting}, for every $\ell_1\in\Lambda(n)$ there exist $\varphi:Y\rightarrow\Sp(2n)$ such that $\varphi(\eta)\ell_1=\ell_2(\eta)$ for all $\eta\in Y$ and $\widetilde{\phi}:X\times Y\rightarrow\Sp(2n)$ such that $\widetilde{\phi}(\lambda,\eta)\ell_1=\gamma(\lambda,\eta)$. Now we set $\phi(\lambda,\eta):=\widetilde{\phi}(\lambda,\eta)\varphi(\eta)^{-1}\in\Sp(2n)$, and obtain

\[\phi(\lambda,\eta)\ell_2(\eta)=\widetilde{\phi}(\lambda,\eta)\ell_1=\gamma(\lambda,\eta),\quad (\lambda,\eta)\in X\times Y.\]
\end{proof}

Our main result of this section reads as follows.

\begin{lemma}\label{famlifting}
Let $X$ and $Y$ be contractible spaces and $\gamma_1,\gamma_2:X\times Y\rightarrow\Lambda(n)$ two continuous maps. Then there exist continuous maps $\phi:I\times X\times Y\rightarrow\Sp(2n)$ and $\ell_1,\ell_2:Y\rightarrow\Lambda(n)$ such that

\[\phi(0,\lambda,\eta)\ell_1(\eta)=\gamma_1(\lambda,\eta),\quad\phi(1,\lambda,\eta)\ell_2(\eta)=\gamma_2(\lambda,\eta),\quad (\lambda,\eta)\in X\times Y.\]
\end{lemma}

\begin{proof}
Since $X$ is contractible, there exists $p\in X$ and a continuous map $f:I\times X\rightarrow X$ such that $f(0,\lambda)=p$ and $f(1,\lambda)=\lambda$ for all $\lambda\in X$. We set $\ell_1(\eta)=\{0\}\times\mathbb{R}^n\in\Lambda(n)$ and take a map $\phi_1:X\times Y\rightarrow\Sp(2n)$ such that $\phi_1(\lambda,\eta)\ell_1(\eta)=\gamma_1(\lambda,\eta)$ for all $(\lambda,\eta)\in X\times Y$. We define $\ell_2(\eta):=\phi_1(p,\eta)^{-1}\gamma_2(p,\eta)\in\Lambda(n)$. By the previous lemma, there exists a map $\widetilde{\phi}_2:X\times Y\rightarrow\Sp(2n)$ such that 

\begin{align}\label{formulalifting}
\widetilde{\phi}_2(\lambda,\eta)\ell_2(\eta)=\phi_1(p,\eta)^{-1}\gamma_2(\lambda,\eta),\quad(\lambda,\eta)\in X\times Y.
\end{align}
Note that by definition $\widetilde{\phi}_2(p,\eta)$ maps $\ell_2(\eta)$ to $\ell_2(\eta)$ so that

\begin{align}\label{formulaell}
\widetilde{\phi}_2(p,\eta)^{-1}\ell_2(\eta)=\ell_2(\eta).
\end{align}
Now we set $\phi_2(\lambda,\eta):=\widetilde{\phi}_2(\lambda,\eta)\widetilde{\phi}_2(p,\eta)^{-1}$ and claim that 

\[\phi:I\times X\times Y\rightarrow\Sp(2n),\quad \phi(t,\lambda,\eta)=\phi_1(f(1-t,\lambda),\eta)\phi_2(f(t,\lambda),\eta)\] has the required properties. Indeed, 

\[\phi(0,\lambda,\eta)\ell_1(\eta)=\phi_1(\lambda,\eta)\ell_1(\eta)=\gamma_1(\lambda,\eta)\]
and from \eqref{formulalifting} and \eqref{formulaell} it follows that

\[\phi(1,\lambda,\eta)\ell_2(\eta)=\phi_1(p,\eta)\widetilde{\phi}_2(\lambda,\eta)\widetilde{\phi}_2(p,\eta)^{-1}\ell_2(\eta)=\phi_1(p,\eta)\widetilde{\phi}_2(\lambda,\eta)\ell_2(\eta)=\gamma_2(\lambda,\eta)\]
for all $(\lambda,\eta)\in X\times Y$.
\end{proof}

Finally, let us point out that all results of this section also hold in the smooth category. For example, if $X$ and $Y$ in Lemma \ref{famlifting} are smooth manifolds and $\gamma_1,\gamma_2$ smooth maps, then the maps $\phi$ and $\ell_1,\ell_2$ can be chosen smooth as well. 

%%%%%%%%%%%%%%%%%%%%%%%%%%%%%%%%%%%%%%%%%%%%%%%%%%%%%%%%%%%%%%%%%%%%%%%%%%%%%%%%%%%%%%%%%%%%%%%%%%%%%%%%%%%%%%%%%%%%%%%%%%%%%%%%%%%%%%%%%%%%%%%%%%%%%%%%%%%%%%%%%%%%%%%%%%%%%%%%%%%%%%%%%%%%%%%%%%%%%%%%%%%%%%%%%%%%%%%%%%%%%%%%%%%%%%%%%%%%%%%%%%%%%%%%%%%%%%%%%%%%%%%%%%%%%%%%%%%%%%%%%%%%%%%%%%%%%%%%%%%%%%%%%%%%%%%%%%%%%%%%%%%%%%%%%%%%%%%%%%%%%%%%%%%%%%%%%%%%%%%%%%

\section{Properties of the spectral flow}
The aim of this section is to show the well-definedness and the homotopy invariance of the spectral flow, where we adapt Phillips' arguments from \cite{Phillips} for bounded operators. We begin by showing the well-definedness.

\begin{lemma}
The right hand side in \eqref{IndPre-align-specflow} depends only on the continuous map $\mathcal{A}$.
\end{lemma}

\begin{proof}
The proof will be divided into three steps.\\
At first, we consider $t_1,\ldots,t_{n-1}$ and $a_1,\ldots,a_n$ as in \eqref{IndPre-align-specflow} and take a further instant $t^\ast\in(0,1)$ such that $t_{i-1}<t^\ast<t_i$ for some $i$. If we now use the two maps

\[[t_{i-1},t^\ast]\ni t\mapsto P_{[-a_i,a_i]}(\mathcal{A}_t),\quad [t^\ast,t_i]\ni t\mapsto P_{[-a_i,a_i]}(\mathcal{A}_t)\]
instead of

\[[t_{i-1},t_i]\ni t\mapsto P_{[-a_i,a_i]}(\mathcal{A}_t)\]
for the computation of \eqref{IndPre-align-specflow}, then the sum does not change because the two new appearing terms cancel each other out.\\
In the second step, we consider the case in which we do not change the partition of the interval but instead the numbers $a_i$. Let $[c,d]\subset[0,1]$ be any subinterval and $t\mapsto P_{[-a_1,a_1]}(\mathcal{A}_t)$, $t\mapsto P_{[-a_2,a_2]}(\mathcal{A}_t)$ two continuous maps as in \eqref{IndPre-align-specflow} which are defined on $[c,d]$. We assume without loss of generality that $a_1\geq a_2$. Since $a_1,a_2\notin\sigma(\mathcal{A}_t)$ for all $t\in[c,d]$, we obtain by \eqref{projdim} 

\begin{align*} 
\dim E_{[0,a_1]}(\mathcal{A}_t)-\dim E_{[0,a_2]}(\mathcal{A}_t)=\dim E_{[a_2,a_1]}(\mathcal{A}_t)=\dim\im P_{[a_2,a_1]}(\mathcal{A}_t)
\end{align*}
which is constant on $[c,d]$ by Corollary \ref{lem:sflneighbourhood}. We conclude that

\begin{align*}
\dim E_{[0,a_1]}(\mathcal{A}_d)-\dim E_{[0,a_1]}(\mathcal{A}_c)&=(\dim E_{[0,a_2]}(\mathcal{A}_d)+\dim E_{[a_2,a_1]}(\mathcal{A}_d))\\
&-(\dim E_{[0,a_2]}(\mathcal{A}_c)+\dim E_{[a_2,a_1]}(\mathcal{A}_c))\\
&=\dim E_{[0,a_2]}(\mathcal{A}_d)-\dim E_{[0,a_2]}(\mathcal{A}_c).
\end{align*}
Finally, let us consider the general case in which we have two partitions $t_0,\ldots,t_n$ and $t'_0,\ldots,t'_m$ having associated numbers $a_1,\ldots,a_n$ and $a'_1,\ldots,a'_m$, respectively, as in \eqref{IndPre-align-specflow}. We build the union of both partitions in order to obtain a third one $\{t''_0,\ldots,t''_{m+n}\}$ which is finer than $t_0,\ldots,t_{n}$ and $t'_1,\ldots,t'_{m}$. By our first step of the proof, we obtain

\begin{align*}
&\sum^{n}_{i=1}{\left(\dim E_{[0,a_i]}(\mathcal{A}_{t_i})-\dim E_{[0,a_i]}(\mathcal{A}_{t_{i-1}})\right)}=\sum^{m+n}_{i=1}{\left(\dim E_{[0,b_i]}(\mathcal{A}_{t''_i})-\dim E_{[0,b_i]}(\mathcal{A}_{t''_{i-1}})\right)}\\
&\sum^m_{i=1}{\left(\dim E_{[0,a'_i]}(\mathcal{A}_{t'_i})-\dim E_{[0,a'_i]}(\mathcal{A}_{t'_{i-1}})\right)}=\sum^{m+n}_{i=1}{\left(\dim E_{[0,b'_i]}(\mathcal{A}_{t''_i})-\dim E_{[0,b'_i]}(\mathcal{A}_{t''_{i-1}})\right)},
\end{align*}
for suitable $b_1,\ldots,b_{m+n}\in \{a_1,\ldots,a_n\}$ and $b'_1,\ldots,b'_{m+n}\in\{a'_1,\ldots a'_{m}\}$. Now the same partition is used in the sums on the right hand sides and we see from the second step of our proof that they actually agree.
\end{proof}

The following assertion is an immediate consequence of the definition of the spectral flow \eqref{IndPre-align-specflow}.

\begin{lemma}\label{IndPre-lemma-N}
Let $N\subset\mathcal{FS}(W,H)$ be a neighbourhood of some $T_0\in\mathcal{FS}(W,H)$ as in Corollary \ref{lem:sflneighbourhood} i). If $\mathcal{A}^1,\mathcal{A}^2:I\rightarrow\mathcal{FS}(W,H)$ are continuous and

\begin{align*}
\mathcal{A}^1(I),\mathcal{A}^2(I)\subset N,\quad \mathcal{A}^1_0=\mathcal{A}^2_0,\quad \mathcal{A}^1_1=\mathcal{A}^2_1,
\end{align*} 
then 

\[\sfl(\mathcal{A}^1)=\sfl(\mathcal{A}^2).\]
\end{lemma}

Now we are ready for proving the homotopy invariance of the spectral flow.

\begin{lemma}
Let $h:I\times I\rightarrow\mathcal{FS}(W,H)$ be a continuous map such that $h(I\times\partial I)\subset GL(W,H)$. Then

\[\sfl(h(0,\cdot))=\sfl(h(1,\cdot)).\]
\end{lemma}

\begin{proof}
Since $h([0,1]\times[0,1])\subset\mathcal{FS}(W,H)$ is compact, we can find a finite open covering

\begin{align*}
h([0,1]\times[0,1])\subset\bigcup^n_{i=1}{N_i},
\end{align*}
where the $N_i\subset\mathcal{FS}(W,H)$ are open sets as in Corollary \ref{lem:sflneighbourhood} i). Accordingly, for each $N_i$ there exists $a_i>0$ such that $-a_i,a_i\notin\sigma(T)$  for all $T\in N_i$, the map 

\begin{align*}
N_i\ni T\mapsto P_{[-a_i,a_i]}(T)\in\mathcal{L}(H)
\end{align*} 
is continuous and all $P_{[-a_i,a_i]}(T)$ are projections of the same finite rank. Let $\varepsilon_0>0$ be a Lebesgue number of the open covering

\begin{align*}
[0,1]\times[0,1]=\bigcup^n_{i=1}{h^{-1}(N_i)},
\end{align*} 
and note that now the image of each subset of $[0,1]\times[0,1]$ of diameter less than $\varepsilon_0$ is entirely contained in one of the $h^{-1}(N_i)$. \\
Let us take instants $0=t_0<t_1<\ldots<t_m=1$ such that $|t_i-t_{i-1}|<\frac{\varepsilon_0}{\sqrt{2}}$, $1\leq i\leq m$. Then for each pair $1\leq i,j\leq m$, $h([t_{i-1},t_i]\times[t_{j-1},t_j])$ is contained entirely in one of the $N_k$. From Lemma \ref{IndPre-lemma-N} and Lemma \ref{IndPre-lemma-sflbasicprop} i)-ii), we obtain for any $h\mid_{[t_{i-1},t_i]\times[t_{j-1},t_j]}$ 

\begin{align*}
\sfl(h(t_{i-1},\cdot)\mid_{[t_{j-1},t_j]})&=\sfl(h(\cdot,t_{j-1})\mid_{[t_{i-1},t_i]})+\sfl(h(t_i,\cdot)\mid_{[t_{j-1},t_j]})\\
&-\sfl(h(\cdot,t_j)\mid_{[t_{i-1},t_i]}).
\end{align*}
Moreover, it follows from Lemma \ref{IndPre-lemma-sflbasicprop} iii) that

\begin{align*}
\sfl(h(\cdot,0)\mid_{[t_{i-1},t_i]})=\sfl(h(\cdot,1)\mid_{[t_{i-1},t_i]})=0,\quad i=1,\ldots,m.
\end{align*}
By using Lemma \ref{IndPre-lemma-sflbasicprop} i) once again, we have

\begin{align*}
\sfl(h(t_{i-1},\cdot))&=\sum^m_{j=1}{\sfl(h(t_{i-1},\cdot)\mid_{[t_{j-1},t_j]})}\\
&=\sum^m_{j=1}{\left(\sfl(h(\cdot,t_{j-1})\mid_{[t_{i-1},t_i]})+\sfl(h(t_i,\cdot)\mid_{[t_{j-1},t_j]})-\sfl(h(\cdot,t_j)\mid_{[t_{i-1},t_i]})\right)}\\
&=\sum^m_{j=1}{\sfl(h(t_{i},\cdot)\mid_{[t_{j-1},t_j]})}=\sfl(h(t_{i},\cdot)),
\end{align*}
and consequently,

\begin{align*}
\sfl(h(0,\cdot))=\sfl(h(t_0,\cdot))=\sfl(h(t_m,\cdot))=\sfl(h(1,\cdot)).
\end{align*}
\end{proof}

\vspace{1cm}
Nils Waterstraat\\
Institut für Mathematik\\
Humboldt Universität zu Berlin\\
Unter den Linden 6\\
10099 Berlin\\
Germany\\
E-mail: waterstn@math.hu-berlin.de

\end{document}